\documentclass[11pt, a4paper]{amsart}
\usepackage{amssymb, amsmath, amsthm, amsfonts, mathtools, multicol, inputenc}

\usepackage{fancyhdr}
\usepackage[a4paper,left=25mm,right=25mm,top=30mm,bottom=30mm]{geometry}

\newtheorem{theorem}{Theorem}[section]
\newtheorem{corollary}[theorem]{Corollary}
\newtheorem{lemma}[theorem]{Lemma}
\newtheorem{proposition}[theorem]{Proposition}

\newtheorem{example}[theorem]{Example}
\newtheorem{problem}[theorem]{Problem}

\theoremstyle{definition}

\DeclarePairedDelimiter{\ceil}{\lceil}{\rceil}
\usepackage[normalem]{ulem}

\title{The Modular Isomorphism Problem: a survey}
\author[L.~Margolis]{Leo Margolis}
 \address{Instituto de ciencias matem\'aticas, C/ Ni\'colas Cabrera 13, 28049 Madrid, Spain}
 \email{{leo.margolis@icmat.es}}

\keywords{Modular Isomorphism Problem, modular group algebras, $p$-groups}
\subjclass[2010]{20C05, 20D15, 16S34, 16L60}
\thanks{The author acknowledges financial support from the Spanish Ministry of Science and Innovation, through the “Severo Ochoa Programme for Centres of Excellence in R\&D (CEX2019-000904-S)}

\begin{document}

\begin{abstract}
The Modular Isomorphism Problem asks if an isomorphism of group algebras of two finite $p$-groups $G$ and $H$ over a field of characteristic $p$, implies an isomorphism of the groups $G$ and $H$. We survey the history of the problem, explain strategies which were developed to study it and present the recent negative solution of the problem. The problem is also compared to other isomorphism problems for group rings and various question remaining open are included.
\end{abstract}

\maketitle

\tableofcontents

\section{Introduction}

Group rings are algebraic objects which come up in many areas of pure and applied mathematics. They were first introduced at the break of the 19th and 20th century in group theory to study groups embedded via representations in groups of matrices, but soon appeared also in topology or functional analysis and are used for instance to generate error correcting codes. But after some time group rings became an object important enough in itself, so that many mathematicians began studying their properties from different perspectives.

To introduce group rings let $G$ be a group and $R$ a ring. Then the \emph{group ring} of $G$ over $R$, denoted as $RG$, is the free $R$-module with basis $G$, i.e.
\[RG = \left\{\sum_{g \in G} r_g \cdot g \ | \ r_g \in R, \ r_g = 0 \text{ almost everywhere } \right\}. \]
Hence, in the case when $R$ is a field we can just think of $RG$ as an $R$-vector space with the basis elements being labeled by the elements of $G$. Then $RG$ has a natural structure of addition, just by adding the coefficients of the corresponding basis elements, but as the basis $G$ has a natural multiplication given by its group structure, we can also define a multiplication on $RG$. For this we let elements of $R$ and $G$ formally commute and extend the multiplication of $G$ associatively and distributively to $RG$. In formulas:
\[\sum_{g \in G}r_g g \cdot \sum_{g \in G}s_g g = \sum_{g \in G} \left(\sum_{h,k \in G, \ hk = g}r_hs_k  \right)g. \]
In this way $RG$ becomes a ring with a group theoretical flavor. Though group rings can be defined for any group and ring, most work is dedicated to the case when $R$ is commutative and we will restrict ourselves to this case in this article.

Many interesting properties of a group ring $RG$ can be studied, a natural starting point being for instance ring-theoretical properties such as for which groups and rings $RG$ is a semisimple ring or which properties the ideals of $RG$ have. Indeed, most early research in the subject was carried out in this direction, as summarized in the first books and comprehensive surveys entirely devoted to group rings \cite{PassmanInfiniteGroupRings, ZalesskiiMihalev73, BovdiBook74, PassmanAlgebraicStructure}. Another view point is to ask in reverse what can be deduced about the group $G$ from the ring structure of $RG$. An easy example would be the order of $G$, which is the $R$-rank of $RG$, or the property of $G$ being abelian, which is the case if and only if $RG$ is commutative. The strongest information one could hope to obtain about $G$ is of course its isomorphism type and this is the content of the following:

\begin{quote}\textbf{The Isomorphism Problem (for group rings):} For $R$ a ring and $G$ and $H$ groups, does an isomorphism $RG \cong RH$ of $R$-algebras imply a group isomorphism $G \cong
H$?
\end{quote}

The first to approach this problem was G.~Higman who studied the group rings of finite groups over the integers in his PhD thesis \cite{Higman1940Thesis} and showed e.g.\ that when $G$ is a finite abelian group, then $\mathbb{Z}G \cong \mathbb{Z}H$ indeed implies that $G \cong H$. Many particular challenging concrete formulations of the Isomorphism Problem were given and solved over time and we review some of the history in the next section. The most famous formulation of the problem which remained open until very recently will be the protagonist for us:

\begin{quote}\textbf{The Modular Isomorphism Problem (MIP):} Let $p$ be a prime, $k$ a field of characteristic $p$, $G$ a finite $p$-group and $H$ some group. Does an isomorphism $kG \cong kH$ imply an isomorphism $G \cong H$? 
\end{quote}

Though the problem had been intensively studied by some authors, positive results remained rather elusive. 
But in the last years it was again attacked independently by several researchers \cite{BaginskiKurdics, Sakurai, MargolisMoede, MargolisStanojkovski21, BrochedelRio21, MargolisSakuraiStanojkovski21} and this culminated in the solution of the problem in its general form:

\begin{theorem}\label{th:counterex}\cite{GarciaLucasMargolisDelRioCounterexample} 
There are non-isomorphic finite $2$-groups $G$ and $H$ such that $kG \cong kH$ for any field $k$ of characteristic $2$.
\end{theorem}

The aim of this survey is to report on the history of the Modular Isomorphism Problem and the methods involved in its study. We will also compare it to similar problems on group rings, especially the Integral Isomorphism Problem, and mention how it connects to other questions on groups and their representations. Though the Modular Isomorphism Problem can be regarded as solved in its generality, many interesting questions remain open and will be mentioned throughout this text. The details on the groups from Theorem~\ref{th:counterex} are given in Section~\ref{sec:Counterex} and a full proof is included.

We pay specific attention to the field appearing in the results. This is a bit cumbersome and thus has usually been avoided by other authors so far, but seems worth in view of the understanding it might add to group algebras in general. Note that in the context of (MIP) the prime field $\mathbb{F}_p$ of $p$ elements provides the strongest restrictions, as from $\mathbb{F}_p$ one can naturally change coefficients for any other field $k$ of characteristic $p$. Formally: 
\begin{align}\label{eq:TensorFp}
\mathbb{F}_pG \cong \mathbb{F}_pH \Rightarrow kG \cong k \otimes_{\mathbb{F}_p}G \cong k \otimes_{\mathbb{F}_p}H \cong kH 
\end{align}
 This even led Passman to state that ``$GF(p)$ is really the only coefficient of field of interest for these groups'' \cite[p. 669]{PassmanAlgebraicStructure}. 

While progress on the Integral Isomorphism Problem had regularly been surveyed in the past, not much was written specifically for the Modular Isomorphism Problem. Some early results and methods were included in \cite{PassmanAlgebraicStructure, Sehgal78}. The last part of \cite{SandlingSurvey} still provides an excellent introduction to the techniques used in its study. Though not officially a survey, the first part of \cite{HertweckSoriano06} provides the most detailed introduction to the problem. Lists summarizing the main results also appeared in \cite{BleherKimmerleRoggenkampWursthorn99} and \cite{EickKonovalov11}. The survey \cite{BovdiUnitsSurvey} gives a detailed overview over results and problems on the unit group of the group algebra of $p$-groups over fields of characteristic $p$.

The article is structured as follows. In Section~\ref{sec:GeneralIPs} we review some history of other isomorphism problems for groups rings, with emphasis on the Integral Isomorphism Problem, and compare them to (MIP). In Section~\ref{sec:History} we mention the first results on (MIP) and list the classes for which a positive solution has been obtained. The following three sections describe some basic methods which are used to study (MIP) starting with Jennings' theory of dimension subgroups and the first consequences derived from it. We elaborate in more detail the idea to embed group bases in quotients of the group algebra in Section~\ref{sec:quotients}. We then prove Theorem~\ref{th:counterex} and mention some consequences of the theorem and questions it leads to. Finally, in Section~\ref{sec:further} we shortly mention several other methods used to study (MIP) and which results they provided so far.

We will mostly be concentrating on the case of group rings of finite groups and when not stated otherwise, we assume the group $G$ to be finite. Moreover, $R$ will always denote a commutative ring with unity $1_R$.
Throughout we use standard group-theoretical notation. The center of a group $G$ is denoted $Z(G)$, the derived subgroup by $G'$, the centralizer of a subgroup $U$ inside $G$ by $C_G(U)$,  the lower central series of $G$ by $(\gamma_i(G))_{i \in \mathbb{N}}$ and the Frattini subgroup by $\Phi(G)$. By $G^p$ we denote the subgroup of $G$ generated by the $p$-th powers of elements in $G$. We write commutators by $[g,h] = g^{-1}h^{-1}gh$ for elements $g$ and $h$ in a group and $[X,Y]$ for the subgroup of $G$ generated by the commutators of $X$ and $Y$ in $G$ where $X$ and $Y$ are subsets of $G$.

We also fix our way of speaking: when $R$ is a field the group ring $RG$ is called a \emph{group algebra}. If moreover the characteristic of $R$ is $p$ and $G$ contains an element of order $p$, we speak of a \emph{modular group algebra}. Furthermore, a group ring over the integers $\mathbb{Z}$ is called an \emph{integral group ring}, a group algebra over the rationals $\mathbb{Q}$ a \emph{rational group algebra} etc.

\section{General Isomorphism Problems}\label{sec:GeneralIPs}

As mentioned, the first result on the Isomorphism Problem was Higman's proof that the isomorphism type of an integral group ring $\mathbb{Z}G$ determines the isomorphism type of $G$ when $G$ is abelian \cite{Higman40}. In the 1950's and 60's many authors compared the isomorphism types of group rings in various special situations. For instance  Perlis and Walker studied group algebras of abelian groups over fields of characteristic not dividing the group order and in particular showed that $\mathbb{Q}G$ determines $G$ up to isomorphism in this case \cite{PerlisWalker}. On the other hand, Berman was the first to exhibit groups with isomorphic rational group algebras, namely the smallest possible example of non-abelian groups of order $p^3$ for $p$ an odd prime \cite{Berman55}. Many other results for special classes of groups and coefficient rings were achieved and several PhD theses written on the topic, which led to the emergence of three main formulations. Apart from the (MIP) these were

\begin{quote}\textbf{The Integral Isomorphism Problem ($\mathbb{Z}$-IP):} Let $G$ be a finite group and $H$ some group. Does an isomorphism $\mathbb{Z} G \cong \mathbb{Z} H$ imply an isomorphism $G \cong H$? 
\end{quote}

\begin{quote}\textbf{(Fields-IP):} Let $G$ be a finite group and $H$ some group such that $KG \cong KH$ holds for every field $K$. Does this imply that $G \cong H$? 
\end{quote}

Apart from being the original problem studied by Higman ($\mathbb{Z}$-IP) was also included, along with (Fields-IP), in \cite[Section 37]{CR0}. Brauer included the general Isomorphism Problem in his influential survey \cite{Brauer63} where he also posed (Fields-IP) as a special problem and (MIP) inside the text.

Though some positive results have been achieved for (Fields-IP) especially for groups of small order, e.g.\ in \cite{Miah8q}, the problem soon found a general negative answer provided by Dade:

\begin{theorem}\label{th:Dade} \cite{Dade71}
Let $p$ and $q$ be primes such that $q \equiv 1 \bmod p^2$. Then there are non-isomorphic metabelian groups of order $p^3q^6$ which have isomorphic group algebras over any field.
\end{theorem}

Recall that a group $G$ is \emph{metabelian}, if it contains a normal abelian subgroup $N$ such that $G/N$ is also abelian. Another pair of groups which have isomorphic group algebras over any field was exhibited by Roggenkamp \cite[Chapter VIII]{RoggenkampTaylor}.

Note that if one does not put any restrictions on the structure of $G$, the formulation ($\mathbb{Z}$-IP) is the strongest one possible: as any ring $R$ is naturally a $\mathbb{Z}$-module, there is a change of coefficients from $\mathbb{Z}$ to $R$ keeping intact isomorphisms between group rings, just as in \eqref{eq:TensorFp}:
\[\mathbb{Z}G \cong \mathbb{Z}H \Rightarrow RG \cong  RH. \]
This is also demonstrated by the fact that already when Dade's counterexamples to (Fields-IP) appeared, it was clear that they do not provide counterexamples to ($\mathbb{Z}$-IP) due to the following result from the thesis of Whitcomb:

\begin{theorem}\cite{Whitcomb} 
Let $G$ be a metabelian group. Then an isomorphism $\mathbb{Z}G \cong \mathbb{Z}H$ implies an isomorphism $G \cong H$.
\end{theorem}

This result was achieved using the so-called small group ring, an object we will also exploit for the (MIP) in Section~\ref{sec:SmallGroupAlgebras}.

Due to the solution of (Fields-IP) the (MIP) and the ($\mathbb{Z}$-IP) were the ``two glimmers of hope'' for ``the ultimate problem'' (i.e. the Isomorphism Problem), as expressed by Passman \cite{PassmanAlgebraicStructure}. Both problems also appeared as conjectures in the detailed group ring survey \cite{ZalesskiiMihalev73}. Indeed, both problems remained open for many years to come and the ($\mathbb{Z}$-IP) was more and more intensively investigated.

A seminal result on the ($\mathbb{Z}$-IP) was achieved by Roggenkamp and Scott, and using other methods also independently by Weiss. Namely, they solved ($\mathbb{Z}$-IP) for all $p$-groups, a notoriously difficult class of finite groups, in a strong form. To explain their statement we first introduce some fundamental notions associated to any group ring $RG$: the map
\[\varepsilon: RG \rightarrow R, \ \ \sum_{g \in G} r_g g \mapsto \sum_{g \in G} r_g \]
is called the \emph{augmentation} of $RG$ and is a homomorphism of rings. An element of $RG$ is called \emph{normalized}, if it has augmentation $1$. We denote by $U(RG)$ the \emph{unit group} of $RG$, i.e. the set of elements $x \in RG$ for which there exists an $x'$ in $RG$ such that $xx' = x'x = 1$. Moreover, $V(RG)$ denotes the subgroup of $U(RG)$ consisting of normalized units and a subgroup of $V(RG)$ which is an $R$-basis of $RG$ is called a \emph{group basis} of $RG$. Note that certainly the group $G$ is a group basis of $RG$ via the embedding $g \mapsto 1_R \cdot g$.

Now an isomorphism $\varphi: RH \rightarrow RG$ of rings will embed $H$ in the unit group $U(RG)$. If we only slightly deform $\varphi$ by defining $\varphi': RH \rightarrow RG$ by the linear extension of $h \mapsto (\varepsilon(\varphi(h))^{-1} \varphi(h)$, then we see that $H$ embeds via $\varphi'$ in $V(RG)$. Hence $H$ is isomorphic to a group basis of $RG$ and the Isomorphism Problem can be translated to the question if all the group bases of $RG$ are isomorphic. In this notion we have:

\begin{theorem}\label{th:RoggenkampScottWeiss}\cite{RoggenkampScott87, Weiss88}
Let $G$ be a finite $p$-group and denote by $\mathbb{Z}_p$ the $p$-adic integers. Then any group basis of $\mathbb{Z}_p G$ is conjugate in $V(\mathbb{Z}_p G)$ to $G$.
\end{theorem}

This result showed that for nilpotent groups actually a stronger statement than $(\mathbb{Z}$-IP) is true, which was known at the time as the second Zassenhaus Conjecture. It asked, if any group basis of $V(\mathbb{Z}G)$ is conjugate inside $V(\mathbb{Q}G)$ to $G$. After the result of Roggenkamp and Scott several more classes of groups were shown to satisfy ($\mathbb{Z}$-IP), cf.\ e.g.\ \cite{RoggenkampTaylor}, but also the first counterexamples to Zassenhaus' conjecture appeared soon \cite{Klingler91}. Finally Hertweck was able to solve ($\mathbb{Z}$-IP) in general:

\begin{theorem}\cite{Hertweck01}
There are non-isomorphic groups $G$ and $H$ of order $2^{21}\cdot 97^{28}$ such that $\mathbb{Z}G \cong \mathbb{Z}H$.
\end{theorem}

It is worth mentioning that these groups were found in a purely theoretical way, not relying on any computer assistance. Though during the decades ($\mathbb{Z}$-IP) was intensively studied many techniques were developed to that end, these are mostly fruitless when one wants to obtain results on (MIP). We will demonstrate this in the next section by exhibiting some key differences between the two problems.

We also note that there is another very active field of algebra which can be viewed as related to the Isomorphism Problem: the study of character degrees, i.e. the understanding of the degrees of irreducible complex representations of $G$ and how the structure of $G$ is related to these degrees. As the degrees, with their multiplicities, is exactly the information provided by the complex group algebra, this can be viewed also as a question on the isomorphism type of $\mathbb{C}G$. There are too many results to give an overview here, but we mention a view. Two classical theorems connecting the structure of $\mathbb{C}G$ and normal subgroups related to Sylow subgroups are the It\^{o}--Michler theorem which states that $G$ has no complex irreducible character of degree divisible by $p$ if and only if $G$ has a normal and abelian Sylow $p$-subgroup \cite{Ito51, Michler86} and the dual theorem of Thompson that when the degree of every non-linear character is divisible by $p$, then $G$ has a normal $p$-complement \cite{Thompson70}. Here $p$ denotes any prime and we should mention that the proofs of the It\^{o}--Michler theorem requires the Classification of Finite Simple Groups. Many variations of these theorems have been studied, we refer to a survey \cite{Navarro16} or a book \cite{NavarroMcKay} by Navarro which describe many of the questions in the area or e.g.\ to \cite{GiannelliRizoSchaefferFry} for a newer result. A question highlighted in \cite[Chapter 7]{NavarroMcKay} is whether the solvability of $G$ can be determined from $\mathbb{C}G$. 
There are also questions in the area directly related to the Isomorphism Problem, e.g the question if a finite simple group is determined by its complex group algebra, which was known as Huppert's conjecture and has been proven in a series of papers by Tong-Viet, cf.\  \cite{TongVietHuppertClassical} for the last step. 

\subsection{Comparing the integral and modular problem}

Some strong restrictions on the group bases of $\mathbb{Z}G$ are provided by the structure of the unit group of $\mathbb{Z}G$ and especially properties of its finite subgroups. The first one was observed by Higman:

\begin{theorem}\cite{Higman40}
If $G$ is abelian and $u \in V(\mathbb{Z}G)$ has finite order, then $u = g$ for some $g \in G$. In particular, $G$ is the only group basis of $\mathbb{Z}G$.
\end{theorem}

This was generalized by Berman:

\begin{theorem}\cite{Berman55Torsion}
If $u \in V(\mathbb{Z}G)$ is of finite order and central in $\mathbb{Z}G$, then $u = g$ for some $g \in G$.
\end{theorem}

Moreover, the order of finite subgroups in $V(\mathbb{Z}G)$ is very much restricted as the order of subgroups in $G$:

\begin{theorem}\cite{ZK}
Let $U$ be a finite subgroup of $V(\mathbb{Z}G)$. Then the order of $U$ divides the order of $G$. Moreover, the elements of $U$ are linearly independent over $\mathbb{Z}$.
\end{theorem}

To understand the grave differences of methods between ($\mathbb{Z}$-IP) and (MIP) it is instructive to study the first non-trivial example:

\begin{example}\label{ex:kC3} 
Let $\mathbb{F}_3 \cong \{0,1,-1\} = k$ be the field with $3$ elements and $G = \langle g \rangle$ the cyclic group of order $3$. Then $kG$ has $3^3 = 27$ elements of which $9$ have augmentation $1$. It is not hard to check that each element of augmentation $1$ is a unit of order $3$ using the identity $(x+y)^p = x^p + y^p$ which holds in any commutative algebra over a field of characteristic $p$. E.g.\
\[ (1 + g - g^2)^3 = 1^3 + g^3 + (-g^2)^3 = 1.\]
As $G$ is the only group of order $3$ up to isomorphism, (MIP) has a positive solution for $G$, but this does not prevent $kG$ from having various group bases. We have $|V(kG)| = 9$, so $V(kG)$ is an elementary abelian group containing $4$ subgroups of order $3$. Three of those are group bases, namely
\[\{1, g, g^2 \},  \ \{1, -1-g, 1-g+g^2 \}, \  \{1, -1-g^2, 1+g -g^2 \}, \]
while $ \{1, -g-g^2, -1 + g + g^2 \}$ is not.
\end{example}

This example shows that none of the three basic facts mentioned before which hold in $V(\mathbb{Z}G)$ are true in the situation of (MIP): the order of the finite group $V(kC_3)$ does not divide $3$ and its elements are not linearly independent, not even for those subgroups whose order does divide the order of $G$. Neither does every central unit in $V(kG)$ lie in the group basis generated by a fixed element of order $3$.

When one views the Isomorphism Problem from the perspective of group bases it is important to understand how a group basis embeds in the normalized unit group of a group ring and how the units act on it. In the integral case this was a problem open for a long time which can be stated as:

\begin{quote}\textbf{Normalizer Problem:} 
In $V(\mathbb{Z}G)$ the normalizer of the subgroup $G$ is generated by $G$ and the center of $V(\mathbb{Z}G)$.
\end{quote}

A counterexample to the Normalizer Problem was the cornerstone of Hertweck's construction of a counterexample to ($\mathbb{Z}$-IP) using a particular unit in $V(\mathbb{Z}G)$ violating the condition of the Normalizer Problem to construct a non-isomorphic group basis. Nevertheless, the Normalizer Problem is still of interest, see \cite{VanAntwerpen}. A very general positive result on the problem was achieved by Jackowski and Marciniak:

\begin{theorem} \cite{JackowskiMarciniak87}
If the Sylow $2$-subgroup of $G$ is normal in $G$, then the Normalizer Problem has a positive answer for $G$. In particular, it has a positive answer for groups of odd order.	
\end{theorem}

In view of the central role of a counterexample to the Normalizer Problem in the construction of counterexamples to ($\mathbb{Z}$-IP) this shows that the following is a wide open problem:

\begin{problem}
Does ($\mathbb{Z}$-IP) hold for groups of odd order?
\end{problem}

In the situation of (MIP) on the other hand the situation with regard to the normalizer of a group basis in the group algebra is very clear. This is due to a result of Coleman \cite{Coleman64} which we cite in a more general form:

\begin{theorem}\label{th:ColemanLemma}\cite[2.1 Lemma]{RoggenkampTaylor}
Let $P$ be a $p$-subgroup of $G$ and $R$ an integral domain in which $p$ is not invertible. Then the normalizer of $P$ in $V(RG)$ is generated by the normalizer of $P$ in $G$ and the center of $V(RG)$.
\end{theorem}

For further facts on finite subgroups of units in $U(\mathbb{Z}G)$ see the survey \cite{MargolisDelRioSurvey}.

Overall, one can say that Sandling's words from \cite{SandlingSurvey} are still true: the perspectives to study ($\mathbb{Z}$-IP) are ``notable by their absence'' in the situation of (MIP). We will dive into the methods used to study (MIP) below, but first add some remarks why it is also of interest, even in weaker forms, to representation theorists.

\subsection{Connections to block theory}
In the representation theory of groups often not the whole group ring $RG$ is the focus of attention, but blocks, i.e. the indecomposable $2$-sided ideals of $RG$. Between blocks not only isomorphisms are of interest, but several other equivalence relations are of importance. These equivalences are usually weaker than isomorphism of rings, such as Morita equivalence or derived equivalence. As the group algebra of a $p$-group $G$ over a field $k$ of characteristic $p$ is indecomposable, the algebra $kG$ is itself a block and hence results on (MIP) can also be regarded from this point of view. Note that in this situation the isomorphism of group algebras is equivalent to Morita and derived equivalence, cf.\ \cite[Proposition 4.3.5]{ZimmermannBook} and \cite[Corollary 2.13]{RouqierZimmermann}.

In this sense the negative answers to (MIP) can also be used to solve questions on blocks. As a more general version of (MIP) Navarro and Sambale considered, whether Morita equivalent blocks necessarily have isomorphic defect groups \cite{NavarroSambale}. Some special positive answers were obtained, while the general question, which is attributed to Alperin in \cite{Bessenrodt90, Scott90}, remained open. The counterexamples to (MIP) also give a negative answer to this question. They also provide a negative answer to a question of Linckelmann, if an isomorphism of blocks in characteristic $p$ can always be lifted to an isomorphism of blocks in characteristic $0$ \cite[Question 4.9]{LinckelmannSurvey}, in view of Theorem~\ref{th:RoggenkampScottWeiss}. We refer to \cite{NavarroGreen, LinckelmannVolII} for text books on the topic.

\section{Early history and overview of result}\label{sec:History}

From now on we will concentrate on (MIP) and if not otherwise stated we assume $G$ to be a finite $p$-group and $k$ a field of characteristic $p$. Some of our remarks will apply to the more general situation of any coefficient ring and we will write $R$ for such rings.

We call a property of $G$ an \emph{invariant} of $kG$, if any other group basis of $kG$ has the same property, e.g.\ the order of $G$ is an invariant. Moreover when we speak of the \emph{class} of $G$, we always mean the nilpotency class. 

From the 1950's on various authors started comparing the isomorphism types of special series of groups or even single examples. The first published result was achieved by Deskins who showed that (MIP) holds for abelian groups \cite{Deskins56}. This was based on the work of Jennings on dimension subgroups which we explain in the next section.  It should be mentioned that also for the group algebras of infinite groups first results in the abelian case were achieved from the 1960's on, in particular Berman solved the problem for countable abelian $p$-groups \cite{Berman67}. Nevertheless, the general case remains open, cf.\ e.g.\ \cite{May14} for more recent results in the area.

Deskins also remarked that the dihedral and quaternion groups of order $8$ have isomorphic group algebras over the field $\mathbb{F}_2$. That this is erroneous was first remarked in \cite{Coleman64} and a more detailed proof was given in \cite{Holvoet66}\footnote{See Deskins' remark in the review of this article on MathSciNet.} by Holvoet. He also handled groups of order $16$ \cite{Holvoet68} and some groups of order $32$ \cite{Holvoet69}. By that time Passman had also proven a positive answer for groups of order $p^4$ for any prime $p$ \cite{Passmanp4}s. A few years later (MIP) was also solved for groups of order $32$ \cite{Makasikis}. Except for the result of Deskins all these answers were achieved over the field with $p$ elements.

In these early papers results were mostly achieved by combining some basic invariants with explicit calculations and an idea to compute numerical invariants under certain maps, which we elaborate on in Section~\ref{sec:maps}. (MIP) was included as a ``widely known'' conjecture in \cite{ZalesskiiMihalev73}, as ``quite possible'' in \cite{PassmanSurvey74}, it was discussed in \cite{BovdiBook74} and was given as Problem 16 in \cite{Sehgal78}. Another approach was taken by Passi and Sehgal who showed how under certain circumstances the problem can be simplified to a problem on a smaller quotient of $kG$ \cite{PassiSehgal72} and this turned out to be a very fruitful idea, cf.\ Section~\ref{sec:quotients}. The price one has to pay though when going down this road is a restriction on the ground field.

Some other ideas were also tried which we shortly sketch in Section~\ref{sec:further}. Nevertheless, the positive results remained rather chaotic and not as strong as in the case of the ($\mathbb{Z}$-IP).
We give an overview over the classes for which the problem has been solved. Here we denote by $D_m(G)$ the $m$-th dimension subgroup of $G$ which is introduced in Section~\ref{sec:DimensionSubgroups}

(MIP) has a positive solution with respect to any field for the following groups:
\begin{itemize}
\item[(i)] Abelian groups \cite{Deskins56},
\item[(ii)] $2$-groups of maximal class \cite{Carlson77} (a module-theoretic proof) / \cite{Baginski92} (a proof in the group algebra),
\item[(iii)] Groups with center of index $p^2$ \cite{Drensky89},
\item[(iv)] Metacyclic groups \cite{BaginskiMetacyclic, SandlingMetacyclic}\footnote{The result is stated only over $\mathbb{F}_p$, but Sandling's arguments work over any field.},
\item[(v)] Groups of order $32$ \cite[Lemma 3.7]{NavarroSambale},
\item[(vi)] $2$-generated groups of class $2$ for $p$ odd \cite{BrochedelRio21}\footnote{The result is stated only over $\mathbb{F}_p$, but the arguments applied in the case of odd $p$ work for any field.}
\item[(vii)] $2$-groups of nilpotency class $3$ such that $[G:Z(G)] = |\Phi(G)| = 8$ \cite{MargolisSakuraiStanojkovski21},
\item[(viii)] $2$-groups with cyclic center such that $G/Z(G)$ is dihedral \cite{MargolisSakuraiStanojkovski21}.
\end{itemize}

Moreover, (MIP) has a positive answer with respect to the prime field $\mathbb{F}_p$ for the following groups:
\begin{itemize}
\item[(ix)] Groups of order at most $p^5$ \cite{Passmanp4, SalimSandlingp5} or $64$ \cite{HertweckSoriano06},
\item[(x)] Groups with $D_3(G) = 1$ \cite{PassiSehgal72} for any $p$ or $D_4(G) = 1$ for $p$ odd \cite{Hertweck07},
\item[(xi)] Groups of maximal class and order at most $p^{p+1}$ which contain a maximal abelian subgroup \cite{BaginskiCaranti88} and most $3$-groups of maximal class \cite{BaginskiKurdics},
\item[(xii)] Groups containing a cyclic subgroup of index $p^2$ \cite{BaginskiKonovalov07},
\item[(xiii)] Groups of class $2$ if additionally $G$ is $2$-generated \cite{BrochedelRio21} or $G'$ is of exponent $p$ \cite{Sandling89},
\item[(xiv)] Groups of class $3$ such that $G'$ has exponent $p$ and additionally either $G$ is $2$-generated \cite{MargolisMoede} or $C_G(G')$ is abelian and a maximal subgroup of $G$ \cite{MargolisStanojkovski21}.
\end{itemize}

Some of these results can be formulated in a more general, but more technical way, which we omit here.

Furthermore, it has been verified by computer calculations that (MIP) has a positive answer with respect to the prime field for groups of order at most $2^8$ or $3^7$ as well as most groups of order $5^6$ \cite{Wursthorn93, BleherKimmerleRoggenkampWursthorn99, Eick08, MargolisMoede}. We include more details on these computer based results in Section~\ref{sec:computers}.

For certain of these results we will see the ideas of the proofs below. Some are just applications of known invariants. We do not include the long list of all the invariants which appear in the literature here, but the interested reader may consult \cite{MargolisMoede}. Most of these invariants are mentioned throughout the text below.

\section{Dimension subgroups}\label{sec:DimensionSubgroups}

The most fundamental tool to study modular group algebras of $p$-groups has its roots in the work of Jennings \cite{Jennings41}. To describe it we first define an ideal that is fundamental for any group ring. Recall that the augmentation of $RG$ is a map sending an element of $RG$ to the sum of its coefficients. The kernel of this map, i.e.\ the elements of augmentation $0$ in $RG$, is called the \emph{augmentation ideal} and denoted by $I(RG)$. Note that in the situation of (MIP) the augmentation ideal is exactly the radical of the group algebra, i.e.\ the collection of elements in $kG$ which annihilate all the simple $kG$-modules. This in turn follows from the fact that the trivial module, i.e.\ the set $k$ on which $G$ acts as the identity, is the only simple $kG$-module. The proof is similar to the fact that the center of a finite $p$-group is non-trivial: if we take $M$ to be a simple $kG$-module and $m$ some non-zero element in $M$, then the $G$-orbit of $m$ spans an additive finite $p$-subgroup in $M$. As the $G$-orbits have $p$-power order and $0$ is a fixed point of the action, there is at least one more fixed point and this shows that the trivial module is a submodule of $M$. The fact that $I(kG)$ is exactly the radical also implies that it is a nilpotent ideal, something which could also be checked by direct calculations.

For any group ring $RG$ and integer $m$ the $m$-th \emph{dimension subgroup} $D_m(G)$ of $G$ with respect to $R$ is the intersection of $G$ and the $m$-th power of the augmentation ideal of $RG$ added to $1$, i.e.\
\[D_m(G) = G \cap (1+I(RG)^m). \]
The dimension subgroups with respect to the integers are rather mysterious objects and the conjecture that in this case the dimension subgroup series coincides with the lower central series of $G$ was a main conjecture in the field in the 1960's and early 70's after this had been shown to be true for free groups by Magnus \cite{Magnus35}. Though this was disproved in general by Rips \cite{Rips72}, the difference between the dimension subgroup series and the lower central series is still not generally understood and new interesting results appeared rather recently \cite{BartholdiMikhailov} (this also contains an overview of other results on the problem).

The vague understanding of dimension subgroups with respect to the integers is again in stark contrast to the situation for the modular group algebra $kG$ of a $p$-group $G$. In this situation Jennings provided a purely group-theoretical description of dimension subgroups which also gives a very handy basis for the group algebra $kG$ \cite{Jennings41}.

Call a series of normal subgroups
\[G = M_1(G) \geq M_2(G) \geq ... \geq M_n(G) = 1 \]
a \emph{$p$-restricted $N$-series}, if for any $i$ and $j$ it satisfies
\[[M_i(G), M_j(G)] \subseteq M_{i+j}(G) \]
and
\[M_i(G)^p \subseteq M_{ip}(G). \]
Note that even when $M_i(G) \neq 1$ the equality $M_i(G) = M_{i+1}(G)$ might hold for some indices $i$. Then we have:

\begin{theorem}\cite{Jennings41, Lazard54} 
For a finite $p$-group $G$ and a field $k$ of characteristic $p$ the dimension subgroup series of $G$ with respect to $k$ is the shortest $p$-restricted $N$-series of $G$. Moreover, one has the inductive formulas 
\[D_m(G) = D_{\ceil{\frac{m}{p}}}(G)^p \gamma_m(G) = \begin{cases}
G & \textup{ if } m=1, \\
[G,D_{m-1}(G)]D_{\ceil{\frac{m}{p}}}(G)^p & \textup{ if } m\geq 2.
\end{cases}\]
and the closed expression
\[D_m(G) = \prod_{ip^j \geq m} \gamma_i(G)^{p^j} . \]
\end{theorem}
This series of subgroups was also constructed, using other ideals, by Zassenhaus and for this reason is sometimes called the \emph{Zassenhaus series} or \emph{Jennings series} or \emph{Brauer--Jennings--Zassenhaus series}.

The dimension subgroup series now allows to exhibit an especially well-behaved basis of $kG$ which respects quotients by powers of the augmentation ideal and thus facilitates calculations. To explain the construction of this basis, note first that by Jennings' theorem for any index $i$ the quotient $D_i(G)/D_{i+1}(G)$ is an elementary abelian group and so can be viewed as an $\mathbb{F}_p$-vector space. Let $g_1,...,g_\ell$ be elements of $D_i(G)$ such that the cosets $g_1D_{i+1}(G),...,g_\ell D_{i+1}(G)$ form a basis of the vector space $D_i(G)/D_{i+1}(G)$. Note that from the definition of dimension subgroups we have that the set $\{g-1 \ | \ g \in D_{i+1}(G) \}$ lies in $I(kG)^{i+1}$. It can then be shown that $g_1-1,....,g_\ell-1$ is a linearly independent set in $I(kG)^i/I(kG)^{i+1}$ and even a basis in case $i=1$. But when we vary the index $i$ we actually obtain much more:

\begin{theorem}\label{th:Jennings} \cite{Jennings41}
Assume $n$ is an integer such that $D_n(G) = 1$. Let $g_1,...,g_\ell$ be the union of the bases of $D_1(G)/D_2(G)$, $D_2(G)/D_3(G)$,...,$D_{n-1}(G)/D_n(G)$ when these quotients are viewed as $\mathbb{F}_p$-vector spaces. Then the set
\[\left\{\prod_{i=1}^\ell (g_i-1)^{\alpha_i} \ | \ 0 \leq \alpha_1,...,\alpha_\ell \leq p-1, \ \sum_{i=1}^\ell \alpha_i \neq 0 \right\} \]
is a basis of the augmentation ideal $I(kG)$. If we add $1$ to this basis, we obtain a basis of the group algebra $kG$.
\end{theorem}

We call a basis as described in Theorem~\ref{th:Jennings} a \emph{Jennings basis} of $kG$. For abelian groups the theorem was obtained a few years earlier by Lombardo--Radici \cite{LombardoRadici}.

From Theorem~\ref{th:Jennings} one finds that the set of dimensions of $I(kG)^i/I(kG)^{i+1}$ and $D_j(G)/D_{j+1}(G)$ determine each other, where $i$ and $j$ run through all possible values. As the augmentation ideal of $kG$ is exactly its radical, and hence defined by a ring-theoretic property independent of $G$, this directly implies:

\begin{corollary}\label{cor:DimQuots}
If $kG \cong kH$, then $D_i(G)/D_{i+1}(G) \cong D_i(H)/D_{i+1}(H)$ for any integer $i$.
\end{corollary}

This invariant led to the first positive result on (MIP), as it was exactly what Deskins used to prove:

\begin{corollary} \cite{Deskins56}
(MIP) holds for abelian groups.
\end{corollary}

In fact, the dimension subgroup series of abelian groups was also considered explicitly by Jennings \cite{Jennings41} and also Lombardo--Radici \cite{LombardoRadici}.

We remark that the quotient $D_1(G)/D_2(G)$ equals $G/\Phi(G)$. This invariant is sometimes attributed to Dieckman who studied it in a more general situation \cite{Dieckman67}. 

The results of Jennings and his successors are not only important in the context of (MIP), but also play a direct role in group theory or cohomology. They are for this reason included with detailed proofs in several textbooks, e.g.\ in \cite[Chapter VIII]{HuppertII}, \cite[Chapter 11]{DdSMS99} or \cite[Section 3.14]{BensonBookI}.

\section{Maps and ideals}\label{sec:maps}

The positive result on (MIP) for abelian groups allows to deduce some of the first known invariants. They both appear already in \cite{Passmanp4}.

\begin{theorem}\label{th:AbAndCent}
The isomorphism types of the abelianization $G/G'$ and the center $Z(G)$ are invariants of $kG$.
\end{theorem}

These basic invariants together with the dimension subgroup quotients from Corollary~\ref{cor:DimQuots} are enough to show that the non-abelian groups of order $p^3$ have non-isomorphic group algebras over any field of characteristic $p$, as observed in \cite{Passmanp4}.

We sketch the proof of Theorem~\ref{th:AbAndCent} and introduce concepts which are very useful in the general study of group rings. For a normal subgroup $N$ of $G$ we call the ideal generated by $\{n - 1 \ | \ n \in N \}$ inside $RG$ the \emph{relative augmentation ideal} of $N$, denoted $I(RN)RG$. Writing the elements of $G$ inside $RG$ with respect to a chosen transversal of $N$ in $G$ it is easy to see that $I(RN)RG$ is the kernel of the linear extension of the natural homomorphism $G \rightarrow G/N$ to $RG$. Hence $RG/I(RN)RG \cong R(G/N)$.  We will also need the \emph{commutator subspace} of $RG$, denoted $[RG,RG]$, which is the $R$-linear space spanned by elements of shape $xy - yx$ for $x,y \in RG$. From the multiplication rule in $RG$ it is clear that $[RG,RG]$ is in fact spanned by elements of type $gh-hg$ for $g,h \in G$.

Now note that the ideal $I(RG')RG$ is the smallest ideal $J$ of $RG$ with commutative quotient: indeed $J$ must contain $g^{-1}h^{-1}gh-1$ for all $g,h \in G$, so it contains $I(RG')RG$. On the other hand as $hg(g^{-1}h^{-1}gh-1) = gh-hg$ for any $g,h \in G$, the ideal $I(RG')RG$ will also contain elements of the shape $xy - yx$ for any $x,y, \in RG$, thus $RG/I(RG')RG$ is commutative. So we can recognize $I(RG')RG$ inside $RG$ by a property defined independently of any group basis, which means that this is also true for its quotient $RG/I(RG')RG \cong R(G/G')$. So in the situation of (MIP) we have found $k(G/G')$, but as (MIP) has a positive answer for abelian groups this means that the isomorphism type of $G/G'$ is an invariant.

To see that also the isomorphism type of $Z(G)$ is an invariant we use a basis for the center of $RG$ given by the class sums of elements in $G$. Recall that for a conjugacy class $C$ in $G$ the \emph{class sum} of $C$ in $RG$ is $\sum_{x \in C} x$. Then it is easy to see that the class sums form an $R$-linear basis of $Z(RG)$, the center of $RG$, which is itself also an $R$-algebra. Note that the elements $gh$ and $hg$ are conjugate in $G$ for any 
$g,h \in G$. As for any $g,h \in G$ also the element $g-h^{-1}gh = h(h^{-1}g) - h^{-1}gh$ lies in $[RG,RG]$, it follows that $[RG,RG]$ consists exactly of those elements in $RG$ the sum of whose coefficients on each conjugacy class vanishes. In the situation of (MIP) the conjugacy class of a non-central element has order a power of $p$ and so the class sum of such a class lies in $[kG, kG]$, while the class sum of a central element clearly does not. It follows that the intersection of $Z(kG)$ with $[kG,kG]$ is spanned by the class sums of non-central elements and that this is an ideal in $Z(kG)$. The quotient is isomorphic to $kZ(G)$ and as (MIP) has a positive solution for abelian groups, we conclude that the isomorphism type of $Z(G)$ is an invariant of $kG$. Note that $Z(G)$ is not determined as a subset of $kG$ though, as demonstrated in Example~\ref{ex:kC3}.

Using the objects $Z(kG)$, $[kG, kG]$ and $I(kG')kG$ Sandling also showed that the isomorphism types of $Z(G) \cap G'$ and $Z(G)/Z(G)\cap G'$ are invariants of $kG$ \cite{SandlingSurvey}. If one takes into account also power maps on these objects his results can be generalized to include more invariants. To state those we introduce the notation $\Omega_n(G) = \langle g \in G \ | \ g^{p^n} = 1 \rangle$ for a non-negative integer $n$. 

\begin{theorem}\label{th:OmegaAndAgemo} \cite{MargolisSakuraiStanojkovski21} For every non-negative integer $n$ and any field $k$ of characteristic $p$ the isomorphism types of the following abelian quotients are invariants of $kG$:
	\begin{multicols}{2}
		\begin{enumerate}
			\item\label{it:invariants1} $G/\Omega_n(Z(G))G'$,
			\item\label{it:invariants2} $\Omega_n(Z(G))G'/G'$,
			\columnbreak
			\item\label{it:invariants3} $Z(G)\cap G'G^{p^n}$,
			\item\label{it:invariants4} $Z(G)/Z(G)\cap G'G^{p^n}$.
		\end{enumerate}
	\end{multicols}
\end{theorem}

Very little is known with regard to reduction of (MIP) to smaller groups. One such reduction can be proven using the invariant $\Phi(G) \cap \Omega_1(Z(G))$ which one gets from Theorem~\ref{th:OmegaAndAgemo}. To state it call a decomposition $G = U \times T$ \emph{elementary} if $T \cap \Phi(G) = 1$ and $|T| = [\Omega_1(Z(G)): \Omega_1(Z(G)) \cap \Phi(G) ]$. Less technically, one can think of $T$ as the biggest elementary abelian direct factor of $G$. We then have:

\begin{theorem}\cite{MargolisSakuraiStanojkovski21}
Assume $kG \cong kH$ and let $G = U \times T$ and $H = V \times S$ be elementary decompositions of $G$ and $H$. Then $T \cong S$ and $kU \cong kV$.
\end{theorem}

It is not known if a similar theorem holds for abelian factors, not even if $k(G \times A) \cong k(H \times A)$ for $A$ an abelian group, does imply $kG \cong kH$.

There are two more invariants connected to the center and power maps which provide us with information on the conjugacy classes of $G$. The first one is due to K\"ulshammer and in particular implies that the exponent of $G$ is an invariant of $kG$ and even of the defect group of a $p$-block \cite{Kuelshammer82}. To explain the part relevant for (MIP) note that for a non-negative integer $n$ and a conjugacy class $C$ in $G$ also the $p^n$-th powers of elements in $C$ form a conjugacy class. Now define a map on $kG/[kG,kG]$ by sending an element $x$ to its $p^{n}$-th power $x^{p^n}$. Then the dimension of the image of this map will be the number of conjugacy classes which are not vanishing after taking the $p^n$-th power, i.e.\ the number of conjugacy classes of $p^n$-th powers.  

A second invariant involving powers of conjugacy classes was discovered by Parmenter and Polcino Milies \cite{ParmenterPolcinoMilies81}. We present the idea as it is reinterpreted in \cite{HertweckSoriano06}. We define the $p^n$-th powers not on $kG/[kG,kG]$, but on $Z(kG)$ and compute again the dimension of the image. It turns out that this is the number of conjugacy classes $C$ in $G$ for which there exists some conjugacy class $D$ such that the $p^n$-th powers of elements in $D$ lie in $C$ and $D$ and $C$ have the same order. Equivalently this is the number of conjugacy classes $C$ so that for $g \in C$ there exists $h \in G$ with $h^{p^n} = g$ and no conjugate $\tilde{h}$ of $h$ satisfies $\tilde{h}^{p^n} = g$. Though in \cite{ParmenterPolcinoMilies81} the authors work with the field of $p$ elements, their proof actually works for any field of characteristic $p$. We summarize:

\begin{theorem}\cite{Kuelshammer82, ParmenterPolcinoMilies81}
The number of conjugacy classes of $G$ which are $p^n$-th powers is an invariant of $kG$ for any non-negative integer $n$. So is the number of conjugacy classes which have the same order as a class which powers to them under the $p^n$-th power map.
\end{theorem}

Especially in the early results on (MIP) also numerical invariants of maps played a major role and we next describe the basic ideas here.

\subsection{Kernel size}
Say the ground field $k$ is the prime field $\mathbb{F}_p$. Then $kG$ is a finite structure and the basis provided by Jennings for the augmentation ideal $I(kG)$ allows to rather simply count the number of elements in $I(kG)$ or some of its quotients which satisfy some condition. E.g.\ one could count the number of elements in $I(kG)/I(kG)^2$ which map to $0$ in $I(kG)^p/I(kG)^{p+1}$ when $x$ is sent to $x^p$. This idea was first described by Brauer which led him to state that ``it might be much easier to study Problem 2 [the Isomorphism Problem] in this particular case [the (MIP) case]'' \cite{Brauer63}. 

To carry out such calculations explicitly one uses the following:
\begin{lemma}\label{lem:BasicIds}
Let $g$ and $h$ be any elements in $G$. Then in $kG$ we have
\begin{align*}\label{eq:BasicFormulas}
(gh-1) &= (g-1) + (h-1) + (g-1)(h-1), \\
(h-1)(g-1) &= (g-1)(h-1) + (1 + (g-1) + (h-1) + (g-1)(h-1))([g,h]-1).
\end{align*}
\end{lemma}
These are easy but fundamental identities which are needed for calculations with a Jennings basis.

\begin{example}
Let $G = \langle a, b \ | \ a^4 = 1, b^2 = a^2, a^b = a^{-1} \rangle$ be the quaternion group of order $8$ and $H = \langle c, d \ | \ c^4 = d^2 = 1, c^d = c^{-1} \rangle$ the dihedral group of order $8$. Set $k = \mathbb{F}_2$. We will show that $kG$ and $kH$ are not isomorphic by counting the number of elements mapping to $0$ under the square-map from $I(kG)/I(kG)^2$ to $I(kG)^2/I(kG)^3$ or from $I(kH)/I(kH)^2$ to $I(kH)^2/I(kH)^3$, respectively.

Abusing notation we will denote images of elements of $I(kG)$ the same way in $I(kG)/I(kG)^2$. First note that by Jennings' Theorem~\ref{th:Jennings} we get that $\{(a-1), (b-1)\}$ forms a basis of $I(kG)/I(kG)^2$.  So $I(kG)/I(kG)^2 = \{0, a-1, b-1, (a-1) + (b-1) \}$. We also can compute that $D_2(G) = \langle a^2 \rangle$ while $D_3(G) = 1$, so that $a^2-1$ lies in $I(kG)^2$ but not in $I(kG)^3$ by Jennings' theorem. Hence the squares of $(a-1)$ and $(b-1)$ are not $0$ in $I(kG)^2/I(kG)^3$. For the last element we use the identities given in Lemma~\ref{lem:BasicIds}:
\begin{align*}
((a-1) + (b-1))^2 &= (a^2-1) + (b^2-1) + (a-1)(b-1) + (b-1)(a-1) \\
&= ( 1+ (a-1)+(b-1)+(a-1)(b-1))(a^2-1) \equiv a^2-1 \bmod I(kG)^3 
\end{align*}
Hence this square is also not $0$ in $I(kG)^2/I(kG)^3$.

On the other hand we get similarly that $I(kH)/I(kH)^2 = \{0, c-1, d-1, (c-1) + (d-1) \}$ and clearly $(d-1)^2 = 0$. Hence $kG$ and $kH$ are not isomorphic.
\end{example} 

This example in fact was the first application of Brauer's idea by Passman in \cite{Passmanp4} who called the numbers computed in this way a \emph{kernel size}. Passman also handled groups of order $p^4$ over the field with $p$ elements. For $p=3$ the kernel size together with the other invariants mentioned so far turned out to be enough, but he needed additional arguments in the other cases. His ideas here are somehow similar to the kernel size, but they use not the augmentation ideal to define maps, but certain special subspaces of the group algebras which are shown to be independent of the chosen group basis, but are defined very particularly using the concrete groups. In fact the groups needing these special arguments are triples of groups of maximal class which are considered separately for the case $p=2$ and $p \geq 5$. The $2$-groups of maximal class of any order were later shown to have non-isomorphic group algebras over any field of characteristic $2$ \cite{Carlson77, Baginski92} while this is still an open question for groups of maximal class for other primes. We note that the proof for groups of maximal class of order $p^5$ is very special and technical, highlighting the difficulties for this class \cite{SalimSandlingMaximal}. Aiming to reprove Passman's result and relying on the other invariants computed by him before the specific subspaces, in the case $p=2$ one could use today the Quillen invariant, cf.\ Section~\ref{sec:Cohomology}, while for $p \geq 5$ the only usable invariant seems to be the isomorphism type of $G/\gamma_2(G)^p\gamma_4(G)$ for $2$-generated groups, which is a rather recent accomplishment, cf.\ Section~\ref{sec:SmallGroupAlgebras}.

The technique of kernel sizes found a few more applications in its generic form using the augmentation ideal and power maps, but not as many as one could have expected from its early successes. It was applied for groups of order $16$ and $32$ \cite{Holvoet68, Holvoet69, Makasikis} and in these cases also often the dimension of the space spanned by the squares of elements in $I(kG)$ was a useful numerical invariant. It was used for a few groups of order $64$ in \cite{HertweckSoriano06} and then again for the case of $p=2$ to finish the proof for $2$-generated groups of nilpotency class $2$ \cite{BrochedelRio21}. In all these cases, it was only with the ground field having $2$ elements. On the other hand, Drensky used a variation of the method (determining if certain kernels are empty or not) to give a proof of (MIP) for groups with center of index $p^2$ over any field \cite{Drensky89} and Bagi\'nski and Caranti used the $p$-power map from $I(kG)/I(kG)^2$ to $I(kG)^p/I(kG)^{p+1}$ to distinguish certain groups of maximal class \cite{BaginskiCaranti88}.

\section{Embedding the group basis in quotients}\label{sec:quotients}

An idea different from the ones described so far is to simplify the search for group bases in $kG$ by considering a quotient $kG/J$, for some ideal $J$ in $kG$, such that any group basis $H$ of $kG$ will map injectively to $kG/J$. Naturally, much depends on the choice of the ideal $J$. This was the approach taken by Whitcomb to solve ($\mathbb{Z}$-IP) for metabelian groups using the ideal $I(\mathbb{Z}G')I(\mathbb{Z}G)$. This ideal was also useful in the study of (MIP) as we show in Section~\ref{sec:SmallGroupAlgebras}, but we start by describing the first idea in this direction.

\subsection{Passi and Sehgal's use of Zassenhaus ideals}\label{sec:LP}

While Jennings considered the series of ideals one gets from the augmentation ideal by powering, the group ring $RG$ carries also a natural Lie structure by the Lie bracket
\[  [\ . \ , \ . \ ]: RG \times RG \rightarrow RG, \ \ [x,y] = xy - yx. \]

This allows to define another filtration of the augmentation ideal given by 
\[I^{[1]}(RG) = I(RG) \ \  \text{and}  \ \ I^{[n]}(RG) = [I^{[n-1]}(RG), I(RG)].\]
By induction it is not difficult to show that by adding the next power of the augmentation ideal this series is connected to the lower central series of $G$ \cite[Lemma 6.8]{Sehgal78}:
\begin{align}\label{eq:Lien-th}
 I^{[n]}(RG) + I(RG)^{n+1} = I(R\gamma_n(G))RG + I(RG)^{n+1} 
\end{align}
for any positive integer $n$. To connect this series to the dimension subgroup series define
\[L_n(RG) = \prod_{ip^j \geq n} (I^{[i]}(RG))^{p^j} + I(RG)^{n+1}.  \]
The $L_n(RG)$ are ideals of $RG$ which were first considered by Zassenhaus \cite{Zassenhaus39} and are for this reason sometimes called \emph{Zassenhaus ideals}. 

Passi and Sehgal realized that in the situation of (MIP) and when one works over the field with $p$ elements these ideals can be used to recognize the $n$-th dimension subgroup of $G$ inside $kG$ up to the following power of the augmentation ideal:

\begin{proposition}\label{prop:Ps72} \cite{PassiSehgal72}
Let $k = \mathbb{F}_p$. Then for any positive integer $n$ one has
\[L_n(kG) = (D_n(G) - 1) + I(kG)^{n+1}. \]
\end{proposition}

We sketch the proof. It involves showing that the map
\[\varphi: D_n(G) \rightarrow L_n(kG)/I(kG)^{n+1}, \ \ g \mapsto g-1 + I(kG)^{n+1} \]
is an epimorphism. To see that assume $i$ and $j$ are such that $ip^j \geq n$ and let $\alpha \in I^{[i]}(kG)$. By \eqref{eq:Lien-th} we can write
\[ \alpha = \sum_{g \in \gamma_i(G)} r_g(g-1) + \beta \]
for some coefficients $r_g \in k$ and $\beta \in I(kG)^{i+1}$. Note that for any integer $\ell$ by Lemma~\ref{lem:BasicIds} we get $(g^\ell-1) = \sum_{t=1}^{\ell} \binom{\ell}{t} (g-1)^t$. Abusing notation we identify an integer and its image in $\mathbb{F}_p$. So for any $g \in D_i(G)$ and $r_g \in \mathbb{F}_p$ we get 
\[r_g (g-1) = (g^{r_g}-1) - \sum_{t=2}^\ell \binom{\ell}{t} (g-1)^t \equiv g^{r_g}-1 \bmod I(kG)^{i+1}. \]
Moreover, also from Lemma~\ref{lem:BasicIds} one has for $g, \tilde{g} \in D_i(G)$ that
\[(g-1) + (\tilde{g}-1) = (g\tilde{g}-1) - (g-1)(\tilde{g}-1) \equiv g\tilde{g}-1 \bmod I(kG)^{i+1}. \]
So, we can rewrite 
\[\alpha = h-1 + \tilde{\beta} \]
for some $h \in \gamma_i(G)$ and $\tilde{\beta} \in I(kG)^{i+1}$. Note that we have used $r_g \in \mathbb{F}_p$ for this, which is the reason we have to restrict the statement of the proposition to the prime field. So, if we take the $p^j$-th power we obtain
\[\alpha^{p^j} = h^{p^j}-1 + \delta \]
for some $\delta \in I(kG)^{n+1}$. Note that $h^{p^j} \in \gamma_i(G)^{p^j} \subseteq D_n(G)$. Looking at the definition of $L_n(kG)$ we see that this implies that $\varphi$ is an epimorphism. The inclusion $D_n(G)-1 + I(kG)^{n+1} \subseteq L_n(kG)$ is easy to prove using Lemma~\ref{lem:BasicIds}, so the proposition follows.

This observations allows to obtain new invariants connected to the dimension subgroup series. We illustrate this by an example:

\begin{example} \label{ex:LieComp}
Let $G = \langle a,b \ | \ a^4 = b^2=1, \ a^b = a^{-1} \rangle$ be the dihedral group of order $8$ and $k$ the field of $2$ elements. Note that $D_2(G) = \langle a ^2 \rangle$ and $D_3(G) = 1$. Let $C$ be a subspace of $I(kG)^2$ containing $I(kG)^3$ such that 
\begin{align}\label{eq:Complement}
I(kG)^2/I(kG)^3 = L_2(kG)/I^3 \oplus C/I(kG)^3, 
\end{align}
i.e.\ $C$ is a complement of $L_2(kG)$ in $I(kG)^2/I(kG)^3$. As $C$ lives between two consecutive powers of the augmentation ideal it is an ideal of $kG$. Note that as $I(kG)/I(kG)^2$ is $2$-dimensional and $L_2(kG)/I(kG)^3$ is $1$-dimensional by Proposition~\ref{prop:Ps72}, the quotient $I(kG)/C$ is $3$-dimensional. From Jennings' theory and again Proposition~\ref{prop:Ps72} we know that the image of $\{1, (a-1) , (b-1), (a^2-1) \}$  in $kG/C$ is a basis of $kG/C$, implying that $V(kG/C) \cong G$.

Now assume that $\varphi: kG \rightarrow kH$ is an isomorphism of rings. As the augmentation ideal is the radical of $kG$ we have $\varphi(I(kG)) = I(kH)$ and this transfers also to powers of the augmentation ideal and the ideals $L_n(kG)$ which are defined only in terms of $I(kG)$. Hence \eqref{eq:Complement} transfers via $\varphi$ to
\[I(kH)^2/I(kH)^3 = L_2(kH)/I^3 \oplus \varphi(C)/I(kH)^3,  \]
i.e.\ the property of being a complement of $L_2$ is a property which remains unchanged under isomorphisms. From Jennings' theory we know that $D_3(H) = 1$ and arguing as before we then get $V(kH/\varphi(C)) \cong H$, implying
\[G \cong V(kG/C) \cong V(kH/\varphi(C)) \cong H. \]
\end{example}

With arguments generalizing this example Passi and Sehgal proved:
\begin{theorem}\label{th:LP72}\cite{PassiSehgal72}
Let $k = \mathbb{F}_p$. Then the isomorphism type of the quotient $D_n(G)/D_{n+2}(G)$ is an invariant of $kG$ for any positive integer $n$. In particular, if $D_3(G) = 1$, then (MIP) has a positive answer for $G$.
\end{theorem}

An improvement of this result using a more delicate generalization of the arguments in Example~\ref{ex:LieComp} was obtained by Furukawa \cite{Furukawa81} and independently by Ritter and Sehgal \cite{RitterSehgal83}.

\begin{theorem} \cite{Furukawa81, RitterSehgal83} 
Let $k = \mathbb{F}_p$. Then the isomorphism type of the quotient $D_n(G)/D_{2n+1}(G)$ is an invariant of $kG$ for every positive integer $n$.
\end{theorem}

Some improvements of this result seem to have been attempted in the 1980's and are mentioned in Sandling's survey \cite[6.22 Theorem]{SandlingSurvey}, but these results were never published.

While here we view the idea of Passi and Sehgal as a first instance of embedding group bases, it could also be understood in the context of normal complements. Note that in Example~\ref{ex:LieComp} we actually show that there is a normal subgroup $N$ of $V(kG)$ such that $V(kG)/N \cong G$. Using that $N$ is defined via a property not depending on the group basis, we then get that any group basis of $kG$ is isomorphic to $G$. We include more details on the generalization of this view in Section~\ref{sec:NormalComplements}. Here we just mentioned that this is essentially the idea used by Hertweck to improve the result of Passi and Sehgal in \cite{Hertweck07}, showing that the isomorphism type of $G/D_4(G)$ is an invariant of $kG$ for $k = \mathbb{F}_p$ and $p$ an odd prime. 

We note that in the counterexamples to (MIP) the smallest $n$ such that $G/D_n(G) \not\cong H/D_n(H)$ is $9$. This leads to:

\begin{problem}
Determine the smallest number $n$ such that $G/D_n(G) \not\cong H/D_n(H)$ holds for $G$ and $H$ two $p$-groups with isomorphic modular group algebras over some field of characteristic $p$. 
\end{problem}

\subsection{Small group algebras}\label{sec:SmallGroupAlgebras}

Note that the invariant $G/D_3(G)$ from the last paragraph equals $G/G^4\gamma_2(G)^2\gamma_3(G)$ in case $p = 2$ and $G/G^p\gamma_3(G)$ in case $p$ is odd. This invariant was generalized by Sandling based on his explicit description of the unit group of the modular group algebra $\mathbb{F}_pG$ for an abelian group $G$ \cite{Sandling84}. This description can also be used to provide an alternative proof of Deskins' result for the field with $p$ elements, as it shows that the unit groups are different for non-isomorphic group bases. But Sandling did not only determine the isomorphism type of the unit group, but provided an explicit independent generating set which is also a main ingredient in the first use of the small group algebra for the study of (MIP).

Let $J = I(kG)I(kG')$ which is clearly an ideal of $kG$. Then $kG/J$ is called the \emph{small group algebra} of $kG$. Note that $I(kG'')$ as well as $I(k(G')^p)$ is contained in $J$, thus we can only hope to obtain information on the quotient $G/\Phi(G')$ when we utilize the small group algebra. Hence, we assume that $\Phi(G') = 1$ to explain Sandling's idea. If we set $S = V(kG/J)$, then we get a natural embedding of $G'$ in $S$ and a short exact sequence
\[ 1 \longrightarrow G' \longrightarrow S\overset{\alpha}{\longrightarrow} V(k(G/G') )\longrightarrow 1.\]
Moreover, $\alpha^{-1}(G/G')$ is isomorphic to $G$ and we will identify $G$ with this preimage. Now for $k = \mathbb{F}_p$ from \cite{Sandling84} we get that the the group $V(kG/G')$ has the shape $G/G' \times C$ and also an explicit independent generating set for $C$. If $A$ is a preimage of $C$, then certainly $S = GA$. The explicit generating set for $C$ also provides us with a generating set for $A$ and hence the structure of $S$ is given quite explicitly. When $\gamma_4(G) = 1$ this structure implies that $A$ is isomorphic to $C$, in particular abelian, and a complement of $G$ in $S$. When $\gamma_3(G)=1$, then the action of $A$ on $G$ is even trivial and so we get:

\begin{theorem}\label{th:Sandling89} \cite{Sandling89}
For $k = \mathbb{F}_p$ the isomorphism type of $G/(G')^p\gamma_3(G)$ is an invariant of $kG$. In particular, (MIP) holds for groups of nilpotency class $2$ with elementary abelian derived subgroup.
\end{theorem} 

But also in the case that $\gamma_3(G) \neq 1$, one still can obtain information on the group bases of $kG$, as they all embed in $S$ (again assuming $\Phi(G') = 1$). So one idea is to study subgroups of $S$ which might come from a group basis and determine whether these are necessarily isomorphic to $G$. This is especially useful when $\gamma_4(G) = 1$ and was one of the main ingredients of Salim and Sandling's proof of (MIP) for groups of order $p^5$ for odd $p$ \cite{SalimSandlingp5}. They showed that for odd $p$ and $G$ a group of order $p^5$ which satisfies $(G')^p\gamma_4(G) = 1$, up to isomorphism $G$ is the only subgroup of $S$ which shares all the known invariants of $kG$ with $G$. This was also shown to be true if $(G')^p\gamma_4(G) = 1$ and $C_G(G')$ is a maximal subgroup of $G$ and abelian, again for $p$ odd \cite{MargolisStanojkovski21}. In the last mentioned paper also some more classes were shown to have these properties which in particular can be used to solve (MIP) for many groups of order $p^6$ and $p^7$. Furthermore, the following invariant can be obtained:

\begin{theorem} 
Let $k = \mathbb{F}_p$ and assume that $G$ is $2$-generated. Then the isomorphism type of $G/(G')^p\gamma_4(G)$ is an invariant of $kG$.
\end{theorem}

This was first noticed in \cite{Baginski99}, though without a proof which was later provided in \cite{MargolisMoede}.

Nevertheless, non-isomorphic groups with isomorphic small group algebras are also known. The first to appear publicly seem to be in \cite{Baginski99}, though there the groups do not satisfy $\gamma_4(G) = 1$. But even when one assumes that $(G')^p\gamma_4(G) =1$ the group $S$ does not determine $G$, not even for groups of order $p^6$, cf.\ the explicit example \cite[Example 3.11]{MargolisStanojkovski21}. So Sandling's theorem cannot be generalized in this case using similar methods and neither can the strategy of Salim and Sandling be extended in this way from groups of order $p^5$ to groups of order $p^6$. We refer to the comments after \cite[Example 3.11]{MargolisStanojkovski21} for more references to examples of isomorphic small group algebras.

A very general reinterpretation of the small group algebra was given by Hertweck and Soriano in \cite{HertweckSorianoFrattini}. They call a group extension
\[1 \longrightarrow N \longrightarrow G \longrightarrow G/N \longrightarrow 1 \]
a \emph{central Frattini extension} if $N$ is elementary abelian, central in $G$ and contained in $\Phi(G)$. They study the isomorphism type of $kG/I(kG)I(kN)$ and compare it to other central Frattini extensions with isomorphic normal subgroup $N$ and quotient group $G/N$. A criterion is obtained which can be used to determine the possible isomorphism of such quotients of group algebras, which is formulated in terms of the automorphism group of $k(G/N)$. Though this criterion could be applied in any given situation, the structure of the automorphism group of a group algebra is rather poorly understood which can make the application complicated in a practical situation. Nevertheless, this might be a way worth exploring. Note that when $G'$ is elementary abelian and $G$ of nilpotency class $2$, then choosing $N = G'$ one gets a central Frattini extension, so that this generalizes Sandling's Theorem~\ref{th:Sandling89}. 

The small group algebra played also a role in the solution of (MIP) in the metacyclic case. A first useful invariant here are the subsequent quotients of dimension subgroups of the derived subgroup:

\begin{proposition}
The isomorphism types of the quotients $D_i(G')/D_{i+1}(G')$ are invariants of $kG$ for any $i$ and field $k$. In particular, the isomorphism type of $G'$ is an invariant if $G'$ is abelian.
\end{proposition}
This was noted in \cite{SandlingSurvey} and also again by Bagi\'nski in \cite{BaginskiMetacyclic}. In the latter paper it was shown that (MIP) has a positive answer over the prime field for metacyclic groups in case $p >3$ by considering $p$-th powers of the augmentation ideal and connecting those to the $p$-th powers in $G$. Later Sandling solved the case of metacyclic groups using techniques which work over any field \cite{SandlingMetacyclic}. Sandling also showed that in case $G'$ is cyclic, the size of the biggest cyclic subgroup of $G$ containing $G'$ is an invariant of $kG$ for $k$ the field with $p$ elements, which could also be used to answer (MIP) for metacyclic groups in this case.

We refer to the original \cite{Sandling89} and also \cite[Section 2.3]{HertweckSoriano06} and \cite{MargolisStanojkovski21} for more details on the small group algebra in the context of (MIP). Moreover, Salim's thesis contains many details on the structure of the unit group of the small group algebra \cite{SalimThesis} some of which are also contained in \cite{SalimSandlingSmallGroupRing}.

\subsection{The concept of Hertweck--Soriano}

A general way of constructing an ideal $J$ of $kG$ such that any group basis of $kG$ embeds in $kG/J$ was described by Hertweck and Soriano for the case $k = \mathbb{F}_p$ and how such an ideal can be used to answer (MIP)  \cite{HertweckSoriano06}. The basic idea consists in first choosing an ideal $J$ in $kG$ with the property that any group basis of $kG$ embeds in $kG/J$. Then one studies the normalized unit group $V(kG/J)$ and determines the conjugacy classes of subgroups in this unit group which might come from a group basis of $kG$, much as it was done in the situation of the small group algebra. If one encounters no subgroup of $V(kG/J)$ which could come from a group basis of $kG$ and is not isomorphic to $G$, then this implies a positive answer to (MIP) in this case. In case one does find a subgroup $H$ of $V(kG/J)$ which is not isomorphic to $G$ and might be coming from a group basis of $kG$, one can consider all the preimages of $H$ in $V(kG)$. If no such preimage turns out to be a group basis of $kG$, then again one obtains a positive answer for (MIP) in this case, but if one finds only one preimage which is a group base, then this provides a counterexample to (MIP).

Of course, much in this strategy depends on how to choose $J$. Note that one can always take $J = 0$, but this just leads to the study of the whole unit group $U(kG)$ which is a notoriously difficult structure, cf.\ Section~\ref{sec:UnitGroup}. The idea for the choice of $J$ described in \cite{HertweckSoriano06} is based on the theory of Jennings and properties of the Zassenhaus ideals, especially Proposition~\ref{prop:Ps72}. Assume $J$ is an ideal such that one has 
\begin{align}\label{eq:IntersectingZassenhaus}
J \cap L_n(kG) \subseteq I(kG)^{n+1} \ \ \text{for all} \ n, 
\end{align}
which is a property defined independently from the group basis $G$. Then any group basis will embed in $kG/J$: indeed, if $g \in D_n(G)$ and $g-1 \in J$, then due to Proposition~\ref{prop:Ps72} we have $g-1 \in J \cap L_n(kG)$ and so $g \in I(kG)^{n+1}$. By Jennings' theory this implies $g \in D_{n+1}(G)$ and as we can continue in this way we obtain $g=1$. 

A practical way to choose $J$ is by starting with a Jennings basis 
\[B = \left\{\prod_{i=1}^\ell (g_i-1)^{\alpha_i} \ | \ 0 \leq \alpha_1,...,\alpha_\ell \leq p-1 \right\}\]
of $kG$ and define a set $\mathcal{A}$ which includes all the elements of $B$ with the sum of the exponents being at least $2$, i.e.\ not $1$ or an element of shape $(g_i-1)$ for any $g_i$. The linear subspace spanned by $\mathcal{A}$ will have the property in \eqref{eq:IntersectingZassenhaus}, again by Proposition~\ref{prop:Ps72}, but this space will in most cases not generate an ideal. On the other hand if $n$ is the minimal integer such that $D_n(G)=1$, then $I(kG)^n$ also has all the properties we demand from $J$ and is also an ideal. Note that by the theory of Jennings it is spanned by a subset of $\mathcal{A}$. Now by trying systematically the elements of $\mathcal{A}$ which lie in $I(kG)^{n-1}$, in $I(kG)^{n-2}$ etc. one can find the maximal subset of $\mathcal{A}$ which can be added to $I(kG)^n$ without destroying the property of being an ideal and satisfying \eqref{eq:IntersectingZassenhaus}. We provide an example:

\begin{example}
Let $G = \langle a,b \ | \ a^8 = b^2 = 1, \ a^b = a^{-1} \rangle$ be the dihedral group of order $16$. A Jennings basis of $I(kG)$ is then given by the set
\[B = \{(a-1)^{\alpha_1}(b-1)^{\alpha_2}(a^2-1)^{\alpha_3}(a^4-1)^{\alpha_4} \ | \ \alpha_i \in  \{0,1\}, \ \alpha_1 + \alpha_2 + \alpha_3 + \alpha_4 \neq 0 \}. \]
The smallest $n$ such that $D_n(G) = 1$ is $n = 5$, so $J$ can include at least the elements of $B$ which lie in $I(kG)^5$. Note that by the theory of Jennings $(a-1), (b-1) \in I(kG) \setminus I(kG)^2$ while $(a^2-1) \in I(kG)^2\setminus I(kG)^3$ and $(a^4-1) \in I(kG)^4 \setminus I(kG)^5$. So the elements of $B$ which are not of the form $(g-1)$ and lie in $I(kG)^m\setminus I(kG)^{m+1}$ are: $(a-1)(b-1)$ for $m = 2$, $(a-1)(a^2-1)$ and $(b-1)(a^2-1)$ for $m = 3$ as well as $(a-1)(b-1)(a^2-1)$ for $m = 4$. 

Clearly we can include $(a-1)(b-1)(a^2-1)$ in $J$ as its product with any element in $I(kG)$ will lie in $I(kG)^5$ and hence in $J$. We cannot include $(a-1)(a^2-1)$ as when multiplied with $(a-1)$ it gives $(a^4-1)$, an element which cannot lie in $J$ due to property \eqref{eq:IntersectingZassenhaus}. Moreover from Lemma~\ref{lem:BasicIds} we get that the product $(b-1)(a^2-1)(a-1)$ will have a summand $(a^4-1)$ when written in the Jennings basis $B$. Thus $(b-1)(a^2-1)$ cannot be included in $J$ either. As $(a-1)(a-1)(b-1) = (b-1)(a^2-1)$, we can also not include $(a-1)(b-1)$ in $J$.

Overall we get an ideal $J$ which is generated by $I(kG)^5$ and the element $(a-1)(b-1)(a^2-1)$. This means that $V(kG/J)$ is generated by the set 
\[\{a,b, 1 + (a-1)(b-1), 1+ (b-1)(a^2-1), 1 + (a-1)(a^2-1)  \}.\]
One can compute using the basic formulas from Lemma~\ref{lem:BasicIds} that this group is presented as:
\begin{align*}
V(kG/J) \cong \langle a,b,c,d,e \ |& \ a^8= b^2 = c^2 = d^2 = e^2 = 1, \ a^b = a^{-1}, \ c^b = cd, \ c^a = ce , \ d^a = da^4, \ e^b = ea^4, \\
                      & e^a = e, \ d^b = d, \ c^d = c, \ c^e = c, \ d^e = d   \rangle 
\end{align*}
\end{example}

This example gives some idea that even when one can construct the ideal $J$, which is in general not that hard for a given concrete group $G$, the unit group $V(kG/J)$ is still a difficult object to investigate and even more are the preimages of possible subgroups of $V(kG/J)$ in $V(kG)$. This strategy was successfully applied for some groups of order $64$ in \cite{HertweckSoriano06}, but the technicality of these applications hints at the fact that this is far from straightforward in general. Nevertheless, the method has a unique advantage compared to the other methods we have described so far: if one succeeds in determining all the subgroups of $V(kG/J)$ which are isomorphic to a given group $H$ of the same order as $G$ and then to determine whether any preimage in $kG$ of such a subgroup is isomorphic to $H$ and linearly independent over $k$, then one has answered whether $kG \cong kH$ holds or not, i.e.\ we get a final answer on (MIP) for a given pair of groups $G$ and $H$ and not just restrictions on possible properties of group bases. Indeed, proceedings along these lines the counterexamples in \cite{GarciaLucasMargolisDelRioCounterexample} were discovered.

It is worth mentioning that from a practical point of view the pure size of the prime $p$ plays a significant role in the application of these ideas. Since a bigger prime will give us, in absolute numbers, a bigger Jennings basis, this will also mean more generators for the group $V(kG/J)$, which complicates manual or even computer assisted calculations.

\subsection{Other quotients}

We shortly mention other quotients of the group algebra in which any group basis with particular properties embeds and which have been used to study (MIP). One was the quotient $kG/I(kG')^2 kG$ which was first used by Bagi\'nski and Caranti \cite{BaginskiCaranti88} and later again by Bagi\'nski \cite{Baginski99}. Note that just as in the case of the small group algebra the ideal $I(kG')^2kG$ contains $I(\Phi(G'))$, so that the maximal quotient of $G$ for which we can obtain information from $kG/I(kG')^2 kG$ is $G/\Phi(G')$. The ideas here are different from those of Sandling and the unit group is not the main tool, though Sandling's description of the unit group of the group algebra of an abelian $p$-group plays a role in \cite{Baginski99}. Rather the idea is to analyze the action of the group algebra on the quotient $I(kG')kG/I(kG')^2kG$. This action is used to show that for $N = C_G(G'/\Phi(G'))$, and under certain additional conditions, the set $I(kN) + I(kG')kG$ is exactly the centralizer of $I(kG')kG/I(kG')^2kG$ in $I(kG)$ and hence a set independent of the chosen group basis. This information is then applied to derive new invariants and new classes for which (MIP) holds. We summarize these:

\begin{theorem}\label{th:BaginskiCarantiBaggi99Invs}
Let $k = \mathbb{F}_p$. The nilpotency class of $G/\Phi(G')$ is an invariant of $kG$ \cite{BaginskiCaranti88}. Moreover set $N = C_G(G'/\Phi(G'))$. If $G/N$ is cyclic then the isomorphism types of $N/\Phi(G')$ and $G/N$ are both invariants of $kG$ \cite{Baginski99}.
\end{theorem}

\begin{theorem}\label{th:BaginskiCarantiBaggi99MIP}
Let $k = \mathbb{F}_p$, then (MIP) has a positive answer for $kG$ if $G$ satisfies one of the following:
\begin{itemize}
\item[(i)] $G$ is a group of maximal class of order at most $p^{p+1}$ such that $G$ contains a maximal subgroup which is abelian \cite{BaginskiCaranti88}.
\item[(ii)] $G$ contains an elementary abelian normal subgroup $N$ such that $G/N$ is cyclic \cite{Baginski99}.  
\end{itemize}
\end{theorem}

Yet another quotient was used by Salim and Sandling to handle those groups of order $p^5$ for which (MIP) remained open after the application of classical invariants and the use of the small group algebra \cite{SalimSandlingMaximal}. These groups turn out to be exactly the groups of maximal class. Note that also in Passman's proof of (MIP) for groups of order $p^4$ the groups of maximal class turned out to be the most complicated case with the methods available to him \cite{Passmanp4}. In \cite{SalimSandlingMaximal} the authors used the quotient $kG/(I(kG)I(kG')^2 + I(kG)^{p+1}I(kG') + I(kG)^{2p+1})$ and its unit group to show that indeed for a group $G$ of maximal class of order $p^5$ the group algebra $kG$ determines $G$ when $k = \mathbb{F}_p$. They also used the result for groups of maximal class given in Theorem~\ref{th:BaginskiCarantiBaggi99MIP}.

\section{The Counterexamples}\label{sec:Counterex}

We describe the counterexamples to (MIP) which have been recently obtained \cite{GarciaLucasMargolisDelRioCounterexample}. As the proof is short and easy to understand, we include it in full.

We start with a very general observation on generating sets of algebras. Denote by $J(A)$ the Jacobson radical of an algebra $A$.

\begin{proposition}\label{prop:AuslanderReitenSmalo}\cite[Proposition 5.2]{KuelshammerBook91} 
Let $A$ be a finite dimensional algebra over a field and $B$ a subalgebra of $A$. Then $A = B + J(A)^2$ implies $A = B$.
\end{proposition}

As the radical of $kG$ equals $I(kG)$, for us this has the following consequence:

\begin{corollary}\label{cor:GenModI2}
Let $g_1,...,g_d$ be a generating set of $G$. Then for any $\alpha_1,...,\alpha_d \in I(kG)^2$ the elements $(g_1 - 1) + \alpha_1,...,(g_d - 1) + \alpha_d$ generate $I(kG)$.
\end{corollary}

Corollary~\ref{cor:GenModI2} could also be proven directly in the modular group algebra of a $p$-group, cf.\ \cite[Section 2]{SalimSandlingSmallGroupRing}. But it is worth noting that the needed property is connected to very general algebra properties of $kG$. 

To define the groups which will give a negative solution to (MIP) let $m$ and $n$ be integers such that $n > m > 2$ and define:
\begin{eqnarray*}
G&=&\langle x,y,z \ \mid \ z=[y,x], \ x^{2^n}=y^{2^m}=z^4=1, \ z^x=z^{-1}, \ z^y=z^{-1} \rangle,	\\
H&=&\langle a,b,c \  \mid \ c=[b,a], \ a^{2^n}=b^{2^m}=c^4=1, \ c^a=c^{-1}, \ c^b=c \rangle.
\end{eqnarray*}

We first show:

\begin{lemma}\label{lem:HAndHNotIsom}
The groups $G$ and $H$ are not isomorphic.
\end{lemma}
\begin{proof}
From the defining relations of $H$ we get $H' = \langle c \rangle$. As $a^2, b \in C_H(H')$ and $\langle c, a^2, b \rangle$ is an abelian and maximal subgroup of $H$ we get $C_H(H') = \langle c, a^2, b \rangle$. The given set is an independent generating set of $C_H(H')$, which follows directly from the defining relations of $H$. So $C_H(H') \cong C_{2^{n-1}} \times C_{2^m} \times C_4$. On the other hand $C_G(G')$ contains an element of order $2^n$, namely $xy$, so $C_G(G') \not\cong C_H(H')$. As the centralizer of the derived subgroup is a characteristic subgroup, we conclude $G \not\cong H$.
\end{proof}

This lemma implies that together with the following we obtain a negative solutions to (MIP). We will use the a basic commutator relation which holds for elements $r,s,t$ in any group:
\begin{align}\label{eq:CommRel}
[r,st] = [r,t][r,s]^t.
\end{align}

\begin{theorem}
In $kH$ the elements $a$ and $b(a+b+ab)c$ generate a group basis isomorphic to $G$.
\end{theorem}
\begin{proof}
We first note that by \eqref{eq:TensorFp} it is enough to handle the case $k = \mathbb{F}_2$, so we assume that $k$ is indeed the prime field.

Set $\tilde{x} = a$ and $\tilde{y} = b(a+b+ab)c$. We first show that these elements generate $kH$. As $c \in H' \subseteq D_2(H)$, we have $c \equiv 1 \bmod I(kH)^2$ by the theory of Jennings. Moreover from $a + b + ab = 1 + (1+a)(1+b)$ we get
\[\tilde{y} = b(a+b+ab)c \equiv b + (1+a)(1+b) \equiv b \mod I(kH)^2.\]
So the elements $\tilde{x}$ and $\tilde{y}$ indeed generate $kH$ by Corollary~\ref{cor:GenModI2}.

Next we show that the elements $\tilde{x}$ and $\tilde{y}$ satisfy the same relations as $x$ and $y$ from the definition of $G$. For this set $\tilde{z} = [\tilde{y},\tilde{x}]$. Clearly the order of $\tilde{x}$ is $2^n$ and $\tilde{x}^2$ is central in $V(kH)$. Together with \eqref{eq:CommRel} this gives
\[1 = [\tilde{y}, \tilde{x}^2 ] = [\tilde{y}, \tilde{x}][\tilde{y},\tilde{x}]^{\tilde{x}} = \tilde{z} \tilde{z}^{\tilde{x}}  \]
which implies $\tilde{z}^{\tilde{x}} = \tilde{z}^{-1}$. We next compute $\tilde{y}^2$. This is facilitated a bit by the fact that $b^2c$ is central in $H$ and hence $\tilde{y} = b^2c + ba(1+b)c$ is the sum of two elements that commute. From the relations $ba = abc$ and $ca = ac^{-1}$ and the fact that $a^2$ is central in $H$ we get
\begin{align}\label{eq:Squarey}
(ba(1+b)c)^2 = a^2b^2c(1+bc)(1+b)
\end{align}
Note that as $b^a = bc$, the set $\{ b,bc \}$ is a conjugacy class in $H$ and so $(1+bc)(1+b)$ is central in $kH$. Hence $(ba(1+b)c)^2$ is a product of central elements, implying that $\tilde{y}^2$ is central itself. This gives by \eqref{eq:CommRel}
\[1 = [\tilde{x}, \tilde{y}^2] = [\tilde{x},\tilde{y}][\tilde{x},\tilde{y}]^{\tilde{y}} = \tilde{z}^{-1}\tilde{z}^{-\tilde{y}} \]
and so $\tilde{z}^{\tilde{y}} = \tilde{z}^{-1}$. Moreover we have from \eqref{eq:Squarey}
\begin{align*}
(ba(1+b)c)^{2^m} &= a^{2^m}b^{2^m}c^{2^{m-1}}((1+bc)(1+b))^{2^{m-1}} = a^{2^m}(1+b^{2^{m-1}}c^{2^{m-1}})(1+b^{2^{m-1}}) \\
 &=  a^{2^m}(1+b^{2^{m-1}})(1+b^{2^{m-1}}) = a^{2^m}(1+b)^{2^m} = 0  
\end{align*}
and so $\tilde{y}^{2^m} = (b^2c)^{2^m} = 1$. The last defining relation we need to check for the group $\langle \tilde{x}, \tilde{y} \rangle$ is that $\tilde{z}^4 = 1$. For this we will use the ideal $J = I(kH')kH$ which was used to prove the invariant $H/H'$ in Theorem~\ref{th:AbAndCent}. The ideal $J$ is the smallest ideal of $kH$ with commutative quotient and hence $1+J$ is the smallest subgroup of $V(kH)$ with abelian quotient, i.e.\ $1+J$ is the derived subgroup. As $\tilde{z}$ is a commutator in $V(kH)$, this gives $\tilde{z} \in 1+J$. Now $H' = \langle c \rangle$, so $J = (c-1)kH$ and $J^4 = (c-1)^4kH = 0$. We conclude that $\tilde{z}^4 \in (1+J)^4 = 1$ and the theorem follows. 
\end{proof}

We note that the groups $G$ and $H$ are remarkably close to several classes of groups for which a positive answer on (MIP) has been obtained. They are $2$-generated of class $3$ with $G'$ cyclic of order $4$, $C_G(G')$ a maximal and abelian subgroup of $G$ and of order at least $2^9$. Moreover they contain a normal $2$-generated abelian subgroup with cyclic quotient. In this sense they are close to being: metacyclic; $2$-generated of class $2$; $2$-generated of class $3$ with elementary abelian derived subgroup; of class $3$ with $C_G(G')$ a maximal and abelian subgroup such that $G'$ is elementary abelian; of order $2^8$. For all these classes a positive solution of (MIP) over the prime field is known, cf.\ Section~\ref{sec:History}.

\section{Non-invariants and possible invariants}
In many sections of this survey we covered invariants of the algebra $kG$. The examples of Section~\ref{sec:Counterex} now allow us to determine also properties of $G$ which are not invariants of $kG$. This is that much easier as the counterexamples to (MIP) are of rather simple shape and small order and hence very accessible to manual or computer aided calculations. For the smallest values $n = 4$, $m=3$ the groups have order $2^9$ and the identifiers $[512,456]$ and $[512,453]$ in the Small Group Library \cite{SmallGroupLibrary} of GAP \cite{GAP}, so one can directly determine many parameters for those groups. For instance:

\begin{proposition}\label{prop:NonInvariants}
The following properties of $G$ are not invariants of $kG$:
\begin{itemize}
\item[(i)] The exponent of $C_G(G')$.
\item[(ii)] The size of the automorphism group of $G$.
\item[(iii)] The number of subgroups of $G$, the number of conjugacy classes of subgroups of $G$ and the number of conjugacy classes of cyclic subgroups of $G$.
\end{itemize}
\end{proposition}

As the number of conjugacy classes of cyclic subgroups in $G$ is exactly the number of indecomposable summands in the algebra $\mathbb{Q}G$ \cite[Corollary 7.1.12]{GRG1}, the last fact has the following consequence:

\begin{corollary}
For $G$ and $H$ as in Section~\ref{sec:Counterex} the rational group algebras of $G$ and $H$ are not isomorphic.
\end{corollary}

Recall that Dade constructed pairs of groups which have isomorphic group algebras over any field, but these were not $p$-groups, cf.\ Theorem~\ref{th:Dade}. The last corollary shows that we still do not know groups of this kind, leading to:

\begin{problem}
Are there non-isomorphic $p$-groups which have isomorphic group algebras over any field?
\end{problem}

We note that as we saw in the proof of Lemma~\ref{lem:HAndHNotIsom} the examples $G$ and $H$ from Section~\ref{sec:Counterex} contain maximal abelian subgroups. Hence the maximal degree of a complex irreducible character of each group is $2$ by \cite[(12.11) Theorem]{Isaacs}. As moreover $G/G' \cong H/H'$ holds by Theorem~\ref{th:AbAndCent} we conclude that the complex group algebras of $G$ and $H$ are isomorphic. We also note that the last point of Proposition~\ref{prop:NonInvariants} has consequences for the question in how far the fusion system of a block is determined by the isomorphism class of this block, cf.\ \cite[Section 8.5]{LinckelmannVolII} for the theory.

Though we now have some non-invariants available, this makes it arguably even more interesting to determine exactly which properties of $G$ are in fact invariants of $kG$. Several fundamental properties of $p$-groups remain mysterious in this regard. One concerns the nilpotency class of $G$. We summarize the knowledge here:

\begin{theorem}\label{th:Class}
The nilpotency class of $G$ is an invariant of $kG$ if one of the following holds:
\begin{itemize}
\item[(i)] The exponent of $G$ is $p$.
\item[(ii)] $G'$ is cyclic.
\item[(iii)] $G$ is of class $2$.
\item[(iv)] $G$ is of maximal class.
\end{itemize}
If $k = \mathbb{F}_p$ and $G'$ is elementary abelian, then the nilpotency class is also an invariant of $kG$.
\end{theorem}

Points (i)--(iii) are consequences of Jennings' theory and the invariants described in Section~\ref{sec:maps}, we refer to \cite{BaginskiKonovalov07} for proofs. Point (iv) was proven by Bagi\'nski and Kurdics \cite[Corollary 3.1]{BaginskiKurdics} and is a consequence of their characterization of a group having maximal class by known invariants of $kG$. Finally, the last point is a consequence of the result of Bagi\'nski and Caranti given in Theorem~\ref{th:BaginskiCarantiBaggi99Invs}. But Theorem~\ref{th:Class} demonstrates that the following remains open:

\begin{problem}
Is the nilpotency class of $G$ an invariant of $kG$?
\end{problem}

While the dimension subgroup series of $G$ is strongly connected to invariants of $kG$ via the theory of Jennings, the connection of the lower central series $(\gamma_i(G))_{i \in \mathbb{N}}$ of $G$, a fundamental series in the study of $p$-groups, remains a secret. Apart from the quotient $G/G' = \gamma_1(G)/\gamma_2(G)$ no other subquotient of the lower central series is known to be an invariant.  The example \cite[Example 2.1]{BaginskiKurdics} shows that $I(k\gamma_3(G))kG$ is not independent of the group basis as a subset in $kG$. 
In \cite[Theorem 4.1]{MargolisStanojkovski21} it is shown that for odd $p$ and $G$ a $2$-generated group of class $3$ such that $\gamma_3(G)$ is of exponent $p$ the isomorphism types of $\gamma_2(G)$ and $\gamma_3(G)$ are invariants of $kG$ for $k = \mathbb{F}_p$. Remarkably the proof of this rather weak invariant seems to need some evolved calculations in certain quotients of $kG$, underlining how bad the connection between $kG$ and the lower central series of $G$ is understood.

An especially interesting subquotient of the lower central series of $G$ with regard to being an invariant of $G$ is $G/\gamma_3(G)$. Proving or disproving this invariant would solve:

\begin{problem}
Does (MIP) hold for groups of nilpotency class $2$?
\end{problem}

Already in the survey \cite{SandlingSurvey} Sandling wrote that it is a ``sad reflection on the state of the modular isomorphism problem that the case of class $2$ groups is yet to be decided in general''. Apart from Theorem~\ref{th:Sandling89}, which is included in \cite{SandlingSurvey}, the only significant improvement for groups of class $2$ seems to be the proof of Broche and del R\'io for $2$-generated groups of this type \cite{BrochedelRio21}. 

Probably the most interesting class of groups for which (MIP) remains open are groups of odd order. At least the most naive attempts to construct counterexamples of odd order as in Section~\ref{sec:Counterex} do not work, leading to:

\begin{problem}
Does (MIP) hold for groups of odd order?
\end{problem}

This includes as a subproblem a formulation which might be more approachable:

\begin{problem}
Does (MIP) hold for groups of exponent $p$?
\end{problem}

\section{Further methods}\label{sec:further}

We briefly mention other perspectives on modular group algebras and methods relevant for (MIP).

\subsection{Cohomology}\label{sec:Cohomology}

A powerful tool to study $p$-groups is given by their cohomology rings. As the $n$-th cohomology group of a $G$-module $M$ which is a $k$-vector space is formally defined as $\operatorname{Ext}^n_{kG}(k, M)$, it is an object solely depending on the group algebra $kG$. As such it is an invariant of $kG$, if one can define the module $M$ independently of $G$.

The most studied case is to let $M = k$ be the trivial module. Some basic properties of the cohomology ring $H^*(G, k)$ are then given as the dimension of the $n$-th cohomology groups. For $n =1$ this is exactly $G/\Phi(G)$, which we already observed to be an invariant in Section~\ref{sec:DimensionSubgroups}, while for $n = 2$ this is the smallest number of relations which is needed to define $G$ as a pro $p$-group \cite[(3.9.5) Corollary]{NeukirchSchmidtWinberg}. This number can indeed contribute new knowledge compared to other known invariants.
If one could systematically test the isomorphism of cohomology rings of two given $p$-groups, this would certainly provide a strong tool for (MIP), though there are groups known which have isomorphic cohomology rings, but for which (MIP) has nevertheless a positive answer. For instance for all groups of order $32$ the cohomology rings were computed by Rusin \cite{Rusin} and some turned out to be isomorphic, though (MIP) had been solved over the field of $2$ elements for groups of order $32$ already at the time \cite{Makasikis}. Later Leary found $3$-groups of order $3^n$, for $n \geq 5$, which have isomorphic cohomology rings over the integers and hence over any ring \cite{Leary3groups}. According to the comment at the end of \cite{SalimSandlingp5} these groups were the original motivation for Salim and Sandling to study (MIP) for groups of order $p^5$. We remark that Leary's groups lie in several classes for which (MIP) was solved later over the field with $p$ elements: they have a cyclic subgroup of index $9$; they have nilpotency class $3$ and elementary abelian commutator subgroup while they are also $2$-generated and the centralizer of the derived subgroup is abelian and a maximal subgroup. So they fall in three classes listed in Section~\ref{sec:History}.

If one can describe the cohomology ring of a $p$-group explicitly over the field of $p$ elements, then via coefficient change one gets it also for other fields: $H^*(G, k) \cong H^*(G, \mathbb{F}_p) \otimes_{\mathbb{F}_p}k$, cf.\ \cite[Section 3.4]{Evens91}. Hence cohomology rings can serve to solve (MIP) over any field. Indeed, while Makasikis had handled groups of order $32$ \cite{Makasikis}, Rusin's work was later used to prove it over any field \cite[Lemma 3.7]{NavarroSambale}. It was also used for $2$-groups $G$ with $[G:Z(G)] = |\Phi(G)| = 8$ \cite{MargolisSakuraiStanojkovski21}, again to avoid using strong invariants which are only available over the prime field. For such applications the algorithm and software to compute cohomology rings written by Green and King and the data base produced from it are very useful, the results being avilable at \cite{GreenKingWebsite} and the algorithm described in \cite{GreenKing11}.

A more systematic approach to (MIP) using cohomology rings and the Massey products therein was taken by Borge \cite{Borge04}. Some of her ideas were used by Ruiz and Viruel to study classifying spaces of $2$-groups of maximal class, in particular yielding an alternative proof of (MIP) over the prime field for those groups \cite{RuizViruel}.

In \cite{Quillen71} Quillen developed a rather general theory for the cohomology ring of a group $G$ over a field of characteristic $p$ and its spectrum of maximal ideals. He showed that it can be stratified by the corresponding spectra of the elementary abelian subgroups of $G$. For us it has the following consequence:

\begin{theorem}
For any integer $m$ the number of conjugacy classes in $G$ of maximal elementary abelian subgroups of rank $m$ is an invariant of $kG$.
\end{theorem}

This is in general a strong invariant. It can be used for instance for the $2$-groups of maximal class.

Any finite-dimensional algebra $A$ carries an $A \otimes A^{\text{op}}$-module structure which allows us to define its $n$-th Hoschild cohomology group $HH^n(A) = \operatorname{Ext}^n_{A \otimes A^{\text{op}}}(A, A)$. For $n = 0$ this is just the center of $A$ \cite[Section 1.2]{WitherspoonHochschildBook}. But already for $n = 1$, this provides new invariants of $kG$. Denote by $Cl(G)$ the set of conjugacy classes of $G$ and by $g^G$ the conjugacy class of an element $g$ in $G$. For group algebras we get a decomposition $HH^n(kG) = \oplus_{g^G \in Cl(G)} H^n(C_G(g), k)$, where the sum runs over the conjugacy classes of $G$ \cite[Section 9.5]{WitherspoonHochschildBook}. The dimension of this space hence gives an invariant which is computable in terms of the cohomology of subgroups of $G$. This is especially handy for $n = 1$:

\begin{theorem}
The number $\sum_{g^G \in Cl(G)} \log_p |C_G(g)/\Phi(C_G(g))|$ is an invariant of $kG$.
\end{theorem}

This invariant is sometimes called the ``Roggenkamp parameter'' and is especially useful for computer aided investigations of (MIP). The first Hochschild cohomology group $HH^1(kG)$ can also be interpreted as the group of outer derivations of $kG$ \cite[Section 1.2]{WitherspoonHochschildBook}. This allows calculations in this group which do not require homological algebra.

With all the powerful machinery developed in the cohomology theory of $p$-groups, it seems certain it could provide more knowledge on properties of group bases of modular group algebras, a hope also expressed in \cite{HertweckSoriano06}. We refer to the many textbooks on cohomology of groups for further details, some of which, in particular, contain Quillen's theory \cite{Evens91, BensonBookII}.

\subsection{Lie methods}

As noted in Section~\ref{sec:quotients} the group algebra carries a Lie algebra structure which was explored for instance by Zassenhaus or Passi and Sehgal. But there are also other Lie algebras associated to $kG$. One such algebra was given by Quillen \cite{Quillen68} based also on the work of Jennings and Lazard. Note that 
\[\text{gr}(kG) = \oplus_{n \geq 1} I(kG)^n/I(kG)^{n+1}\]
is a graded restricted Lie algebra. Recall that $[D_n(G), D_m(G)] \subseteq D_{n+m}(G)$ and $D_n(G)^p \subseteq D_{np}(G)$, so that we can equip $\oplus_{n \geq 1} D_n(G)/D_{n+1}(G)$ with the structure of a restricted Lie algebra as well, the Lie bracket being given by commutators and the $p$-map by taking the $p$-th power. This algebra embeds in $\text{gr}(kG)$ by sending an element $g \in D_n(G)/D_{n+1}(G)$ to $g-1 \in I(kG)^n/I(kG)^{n+1}$. It turns out that $\text{gr}(kG)$ is the universal enveloping algebra of the image of this map and moreover this image coincides with the subalgebra generated by the elements of degree $1$. Hence, the isomorphism type of $\oplus_{n \geq 1} D_n(G)/D_{n+1}(G)$ as a restricted Lie algebra is an invariant of $kG$. 

This invariant has not been used much to study (MIP). In Salim's thesis an example is included of how it can be used for certain groups of order $3^5$ \cite{SalimThesis}, as also mentioned in \cite{SalimSandlingp5}. One can understand the most common application of the kernel size method described in Section~\ref{sec:maps} as a calculation in $\text{gr}(kG)$, just using the power and not the Lie structure of this algebra. In \cite[Theorem 5.8]{MargolisStanojkovski21} the algebra $\oplus_{n \geq 1} D_n(G)/D_{n+1}(G)$ is used to provide an easy argument that groups coming from two natural subfamilies, in a class of 2-generated groups of unbounded class, cannot have isomorphic group algebras.

Another Lie structure associated to a modular group algebra is given by Hochschild cohomology groups. Defining $HH^*(kG) = \oplus_{n \geq 0}HH^n(kG)$ this structure is a Gerstenhaber algebra and in particular a graded Lie algebra of degree $-1$, as shown by Gerstenhaber \cite{Gerstenhaber} (cf.\ also \cite[Section 1.4]{WitherspoonHochschildBook}). This implies that $HH^1(kG)$ is a Lie algebra by itself and when we view $HH^1(kG)$ as the outer derivations of $kG$ the Lie bracket is just $[f,g] = f \circ g - g \circ f$. 
For group algebras the Hochschild cohomology groups are also Batalin--Vilkovisky algebras, as shown by Tradler \cite{Tradler08} and the Gerstenhaber structure can be derived from the Batalin-Vilkovisky algebra and the cup product. 

These structures have not been used much to obtain properties of $G$ which might also be due to the fact that explicit calculations of Hochschild cohomology groups are very cumbersome, see e.g.\ \cite{SanchezFlores}. But recently some connections have been established which can be seen as leading to invariants of $kG$, though so far these seem to include no new invariants. E.g.\ Benson, Kessar and Linckelmann established a connection between $HH^1(kG)$ and $Z(G)/(Z(G) \cap \Phi(G))$ \cite{BensonKessarLinckelmann}, an invariant included in Theorem~\ref{th:OmegaAndAgemo}, while Briggs and Rubio y Degrassi showed that the maximal toral rank of $HH^1(kG)$ coincides with the rank of $G/\Phi(G)$  \cite[Proposition 5.12]{BriggsRubioYDegrassi}, an invariant following from the theory of Jennings. Note that in the situation of (MIP) we always have $HH^1(kG) \neq 0$, so that the first Hochschild cohomology will contain at least some information on $G$. This follows as $kG$ always has a non-trivial outer derivation, cf.\ the proof of \cite[Theorem 2.2]{FarkasGeissMarcos}.

As with the cohomology of $p$-groups, the theory of Lie algebras should have the potential to establish new interesting connections between the structure of $kG$ and $G$, though the Lie algebras themselves are not that easy to access either. They have played a significant role though in the investigation of the unit group $U(kG)$, cf.\ e.g.\ \cite{Shalev91}.

\subsection{Computer algorithms}\label{sec:computers}

As (MIP) studies a finite object, when one assumes that $k$ is finite, it is natural to use computers for its investigation. Indeed, at least from the 1970's on this has been done and references appear in \cite{MoranTench77, Ivory80} and later in several articles of Sandling who also references unpublished work of other authors \cite{Sandling84, Sandling89}. In these early stages the results achieved usually concerned properties or even full descriptions of the unit group of $kG$. Though positive results on (MIP) for groups of order $3^5$, $3^6$ or even $p^5$ for all primes are mentioned, these have never been published.
Coleman wrote a program which allowed him to calculate the whole multiplication table of $kG$ and perform experiments with units or ideals \cite{Coleman92}.

A first algorithm that could in principle settle (MIP) for any given pair of groups was designed by Roggenkamp and Scott and implemented with improvements by Wursthorn in his Diplomarbeit \cite{Wursthorn90}. The algorithm, which is also described in \cite{Wursthorn93}, allows for $k = \mathbb{F}_p$ to compute the ring homomorphisms from $kG$ to $kH$ or rather to $kH/I(kH)^n$ for an increasing integer $n$. This algorithm was first successfully applied for groups of order $2^6$ \cite{Wursthorn93} and later for groups of order $2^7$ \cite{BleherKimmerleRoggenkampWursthorn99}. The implementation is not available nowadays anymore.

A second algorithm based on the ideas underlying also the algorithm which classifies groups of small order was developed by Eick \cite{Eick08}. This algorithm allows to compute a normal form of the multiplication table of a finite nilpotent algebra in prime characteristic and hence the algebras $I(kG)/I(kG)^n$ and $I(kH)/I(kH)^n$ can be compared for increasing $n$. Thus, it can also solve, in principle, (MIP) for any given pair of groups over a finite field. The algorithm was implemented by Eick in GAP \cite{ModIsom} and used to solve (MIP) over the prime field for groups of order $2^8$, $3^6$ and $3^7$. It was also used to solve it for groups of order $2^9$ in a big parallel computation \cite{EickKonovalov11}. However, it was later discovered that the implementation of the algorithm contained a programming error, so that it became necessary to recompute these results \cite{MargolisMoede}. This has been done, along with an improvement of the algorithm \cite{ModIsomExt}, for the orders $2^8$, $3^6$ and $3^7$, but as the correction of the error significantly increased the run time and memory usage of the algorithm, it was decided not to systematically recompute the groups of order $2^9$. After the discovery of the examples in Section~\ref{sec:Counterex} it was however checked that among the $2$- and $3$-generated groups of order $2^9$ these are the only negative solutions of (MIP) \cite{MargolisMoede}. 

Though the algorithms of Wursthorn and Eick allow in principle to solve (MIP) for any given pair of groups, it was already foreseen by Coleman (commenting on Wursthorn's implementation) that its ``disadvantage is that the amount of time [...] might be prohibitive''. Indeed, to solve (MIP) for a certain pair of groups of order $5^6$ in \cite{MargolisMoede} a computation was performed which ran for about 4 months and needed 200GB of memory (which turned out not to be enough for other cases of that order). For these reasons in a systematic study of (MIP) for small orders usually first group-theoretical invariants of $kG$ are computed and only to those groups which share all these invariants the algorithms are applied, we refer to \cite{Eick08, MargolisMoede} for more details. 

Experiments in the style performed by Coleman in the early 1990's became open to everybody a decade later thanks to the GAP package LAGUNA \cite{LAGUNA} which allows to perform computations in finite group algebras rather efficiently. The package UnitLib \cite{UnitLib} contains a library of unit groups of modular group algebras of $p$-groups for groups of small order and can also be used to obtain new ideas.

Moreover, the solution of (MIP) for most $2$-groups of almost maximal class by Bagi\'nski and Konovalov was aided by computations of invariants for various groups of small order  \cite{BaginskiKonovalov04}.

\subsection{Automorphisms and profinite approach} 

The general Isomorphism Problem or related problems can also be regarded as questions on the automorphism group of a group algebra. While for $(\mathbb{Z}$-IP) this was a fruitful approach, this was less so for (MIP) and very little is known in general on $\operatorname{Aut}(kG)$. The first result seems to be that as soon as the order of $G$ is bigger than $2$ the algebra $kG$ has an outer automorphism, i.e.\ one not induced by conjugation with an element of $U(kG)$. This was observed by Carns and Chao \cite{CarnsChao72}. Later Bagi\'nski made some special observations on the automorphism of $kG$ for $G$ a $2$-group of maximal class \cite{Baginski92}.

A more systematic approach was taken by R\"ohl \cite{Roehl87, Roehl90}. He showed that if $G$ is a $d$-generated group and $A_d$ the completed group algebra of the free group on $d$ generators, then an automorphism of $kG$ can be lifted to an automorphism of $A_d$ and this is also true for unitary automorphisms. Here an automorphism of $kG$ is called \emph{unitary}, if it induces the identity of $I(kG)/I(kG)^2$. As any power of the augmentation ideal is invariant under automorphisms of $kG$, using the connection of dimension subgroups and powers of the augmentation ideal R\"ohl proved:

\begin{theorem}\cite{Roehl90}
Let $k = \mathbb{F}_p$ and $F$ a free group. Let $R$ be a normal subgroup of $F$ such that for some $n$ we have $D_{n+1}(F) \subseteq R \subsetneq D_n(F)$. Then an isomorphism $k(F/R)/I(k(F/R))^{n+1} \cong kG/I(kG)^{n+1}$ implies an isomorphism $F/R \cong G$. In particular, (MIP) holds for $F/R$. 
\end{theorem}

This can also be seen as a generalization of Theorem~\ref{th:LP72} and the proof also uses Zassenhaus ideals and their complements in quotients modulo powers of the augmentation ideal. R\"ohl moreover showed that under certain conditions the automorphism group of $kG$ can be written as a product of the automorphisms of $G$ and the unitary automorphisms of $kG$. Note that the isomorphism one gets from the examples in Section~\ref{sec:Counterex} is in fact unitary.
We refer to \cite[Chapter 7]{WilsonBookProfinite} for basic definitions and facts on completed group algebras. It was shown though by Hertweck and Soriano how R\"ohl's results can also be achieved in a purely finite context \cite{HertweckSorianoFrattini}. They also obtain results on the automorphisms of $kG$ for $G$ elementary abelian or a dihedral group of order $8$.


Eick's algorithm, which was mentioned in Section~\ref{sec:computers} allows to compute $\operatorname{Aut}(kG)$ explicitly. However, the commands of the package \cite{ModIsom} do return the automorphism group not with respect to the group basis $G$, but to the normal form of $kG$ which is computed in parallel. This makes it harder to use it for practical investigations.

An idea to go the reverse way and study outer automorphisms of $G$ using the group algebra was suggested by Bagi\'nski: assume $J$ is an ideal of $kG$ such that $G$ maps injectively to $kG/J$. Then one could study the normalizer of $G$ in the unit group of $kG/J$ and obtain information on outer automorphisms of $G$ in this way. Note that when $J = 0$ this is not possible by Theorem~\ref{th:ColemanLemma}.

\subsection{The unit group}\label{sec:UnitGroup}

As the Isomorphism Problem can be formulated in terms of the group bases contained in $RG$, i.e.\ certain subgroups of $U(RG)$, one approach to it could be to study the subgroup structure of $U(RG)$. While in the case of integral group rings this strategy had some success, cf.\ Section~\ref{sec:GeneralIPs}, in the case of (MIP) this has been less productive. Note that if $k = \mathbb{F}_p$, the normalized unit group $V(kG)$ has order $p^{|G|-1}$. For some groups $G$ this unit group can be explicitly described and we saw in Section~\ref{sec:SmallGroupAlgebras} how Sandling's description of an explicit independent generating set in the abelian case was useful to solve (MIP) for new classes. Later Sandling also described the unit group for $G$ of order $16$ \cite{Sandling92Order16}. An algorithm of A.~Bovdi~\cite{Bovdi98} was used by Konovalov and Yakimenko to provide a library of the normalized unit groups over the prime field for any $p$-group of order smaller than $243$ \cite{UnitLib}. This library can be used to generate new ideas on the unit group.

But in general the structure of $V(kG)$ is too difficult to provide detailed insight on $G$. Though it has been investigated more intensively than (MIP), even basic parameters of $V(kG)$, like its exponent or its nilpotency class, are not understood in general, see e.g.\ \cite{Shalev91}. This is a reason for the strategy described in Section~\ref{sec:quotients} which tries to simplify the structure by considering quotients of $kG$. For an overview on many problems which have been studied for $U(kG)$ we refer to the survey \cite{BovdiUnitsSurvey}. We are not aware of any newer survey, though many new results have been achieved since.

A question attributed to Berman asked if an isomorphism of unit groups $U(kG) \cong U(kH)$ will imply an isomorphism $G \cong H$ and has been positively answered for some cases, e.g.\ in \cite{BaloghBovdi04}. The examples from Section~\ref{sec:Counterex} give also negative answers to this question. A special subgroup of units in $kG$ which has found some attention in the literature are so-called \emph{unitary units}, these are elements $u \in V(kG)$ such that $u^{-1} = u^*$ where $*$ denotes the linear extension of the inversion map $g \mapsto g^{-1}$ to $kG$. This is a subgroup smaller than $V(kG)$, so it might be easier to investigate, but with regard to (MIP) it is not clear what happens to this subgroup when group bases are changed. See e.g.\ \cite{BaloghBovdi20, Balogh21}.

\subsection{Normal complements}\label{sec:NormalComplements}

A question on the unit group of $RG$ especially connected to the isomorphism problem is the existence of a normal complement to the group $G$, i.e.\ when does $V(RG) = N \rtimes G$ hold for some normal subgroup $N$ of $V(RG)$. In the situation of $(\mathbb{Z}$-IP) the existence of a torsion free normal complement gives a positive solution to the isomorphism problem, cf.\ \cite[Chapter 4]{Sehgal93}. In the modular group algebra of a $p$-group we do not have elements of infinite order, but still one might hope that the existence of a complement could provide a positive answer to (MIP).  This was claimed in \cite{Roehl91}, but the proof is know to contain an error, so that this question remains unanswered. Of course, if one can construct a complement in a manner which does not depend on the group basis $G$, then one obtains a positive solution for (MIP).

We give a brief account of what is concretely known on the question. All results are for $k = \mathbb{F}_p$.
Several early papers studying the existence of normal complements showed that a complement exists for abelian groups, which is also a consequence of Sandling's work in \cite{Sandling84}, and also for some non-abelian groups of small order \cite{MoranTench77, Johnson78, Ivory80}, but there is no complement for dihedral, semidihedral or quaternion groups of order at least $16$ \cite{Ivory80}. The work of Passi and Sehgal described in Section~\ref{sec:LP} implies that a complement exists when $D_3(G) = 1$ and the work of Sandling on small group algebras, as explained in Section~\ref{sec:SmallGroupAlgebras}, implies the existence also in case $(G')^p\gamma_3(G) = 1$. Moreover by work of R\"ohl a complement exists for $G$ a unit group of an $\mathbb{F}_p$-algebra or a normal subgroup of unitriangular matrices \cite{Roehl91}. Hertweck's proof of (MIP) for groups satisfying $D_4(G)=1$ for $p$ odd follows from the construction of a complement also in this case. In \cite{KaurKhan19} a complement is exhibited for split metacyclic $p$-groups of class $2$, which are also covered by Sandling's result. In \cite{KaurKhan20} the existence of a complement is claimed for split metabelian groups of exponent $p$, but the proof of this result contains an error. We refer to \cite{Sandling89} where some other early results are mentioned.

A new approach to (MIP) which can be viewed also from the perspective of normal complements was given by Sakurai \cite{Sakurai}. He showed that (MIP) holds for so-called hereditary groups over $k$. We will define his concept in a less technical manner. Namely call a group $H$ a \emph{virtual unit group} over a ring $R$ if there are finite $R$-algebras $A$ and $B$ such that the unit group of $A$ is isomorphic to the direct product of $H$ and the unit group of $B$. Then $G$ is \emph{hereditary} over $R$ if the direct product of finitely many copies of each subgroup of $G$, including $G$ itself, is a virtual unit group over $R$. This criterion can be used to reprove for instance (MIP) for groups satisfying $D_3(G)=1$ over the field of $p$ elements. To apply it more broadly one would need to understand which groups appear as unit groups of finite $k$-algebras. Now if $G$ has a normal complement $N$ in $V(kG)$ such that $N-1$ is an ideal of $kG$, then $G$ is the unit group of $kG/(N-1)$. This calls the attention to answer not only whether $G$ has a normal complement in $V(kG)$, but whether this complement is coming from an ideal, a question raised also in \cite[Section 5]{Bovdi98}. Most of the results cited in the previous paragraph come from ideals, but  in \cite[Section 4.3]{HertweckSoriano06} an example is given where this is not the case and it is not hard to check that the explicit example given in \cite{Hertweck07} is not coming from an ideal either. The question of which groups come up as unit groups of $k$-algebras can also be approached using results on which groups come up as unit groups of rings, also known as Fuchs' Problem. Though to combine results with Sakurai's criterion one would rather need to study normalized unit groups, which need to be properly defined for general rings, but at least in characteristic $2$ results such as \cite{SwartzWerner}\footnote{A erroneous theorem relevant for an application of Sakurai's criterion has been corrected in the ArXiv version of the paper.} could be directly applied. 

\vspace*{1cm}

\textbf{Acknowledgment:} I am grateful to all my colleagues with whom I have studied various aspects of the Modular Isomorphism Problem: Diego Garc\'ia Lucas, Tobias Moede, \'Angel del R\'io, Taro Sakurai, Mima Stanojkovski and Matteo Vannacci. I also thank Czes\l{}aw Bagi\'nski for answering many of my questions. Moreover, the comments of Bagi\'nski, Garc\'ia Lucas, Sakurai and Stanojkovski as well as Burkhard K\"ulshammer and Lleonard Rubio y Degrassi have been very helpful in improving this exposition.

\bibliographystyle{amsalpha}
\bibliography{MIP}

\end{document}